\documentclass[reqno, 12pt]{amsart}
\pdfoutput=1
\makeatletter
\let\origsection=\section \def\section{\@ifstar{\origsection*}{\mysection}} 
\def\mysection{\@startsection{section}{1}\z@{.7\linespacing\@plus\linespacing}{.5\linespacing}{\normalfont\scshape\centering\S}}
\makeatother        

\usepackage{amsmath,amssymb,amsthm}
\usepackage{mathrsfs}
\usepackage{mathabx}\changenotsign
\usepackage{dsfont}

\usepackage{xy}
\xyoption{all}
\usepackage{graphicx}
 
\usepackage{xcolor}
\usepackage[backref]{hyperref}
\hypersetup{
    colorlinks,
    linkcolor={red!60!black},
    citecolor={green!60!black},
    urlcolor={blue!60!black}
}

\usepackage{mathtools}
\mathtoolsset{showonlyrefs}

\usepackage[abbrev,msc-links,backrefs]{amsrefs} 
\usepackage[open,openlevel=2,atend]{bookmark}

\usepackage{doi}

\renewcommand{\PrintDOI}[1]{\doi{#1}}

\usepackage[T1]{fontenc}
\usepackage{lmodern}
\usepackage[babel]{microtype}

\usepackage[english]{babel}

\usepackage{geometry}
\usepackage{enumitem}
\def\rmlabel{\upshape({\itshape \roman*\,})}
\def\RMlabel{\upshape(\Roman*)}

\let\polishlcross=\l
\def\l{\ifmmode\ell\else\polishlcross\fi}

\let\setminus=\smallsetminus

\makeatletter
\def\moverlay{\mathpalette\mov@rlay}
\def\mov@rlay#1#2{\leavevmode\vtop{   \baselineskip\z@skip \lineskiplimit-\maxdimen
   \ialign{\hfil$\m@th#1##$\hfil\cr#2\crcr}}}
\newcommand{\charfusion}[3][\mathord]{
    #1{\ifx#1\mathop\vphantom{#2}\fi
        \mathpalette\mov@rlay{#2\cr#3}
      }
    \ifx#1\mathop\expandafter\displaylimits\fi}
\makeatother

\let\eps=\varepsilon
\let\epsilon=\varepsilon
\let\theta=\vartheta
\let\rho=\varrho
\let\phi=\varphi

\def\({\left(}
\def\){\right)} 
\def\<{\left\langle} 
\def\>{\right\rangle} 
\def\llfloor{\left\lfloor}
\def\rrfloor{\right\rfloor}
\def\llceil{\left\lceil}
\def\rrceil{\right\rceil}
\let\epsilon=\varepsilon
\let\rho=\varrho
\def\colond{\colon}
\let\:=\colond
\let\eps=\epsilon
\def\e{{\rm e}}
\def\I{{\rm i\,}}
\def\chia{\chi_{\textsc{arg}}}
\def\Re{\mathop{\textrm{\rm Re}}\nolimits}
\def\Im{\mathop{\textrm{\rm Im}}\nolimits}

\def\im{\mathop{\textrm{\rm im}}\nolimits}

\def\DISC{\mathop{\textrm{\rm DISC}}\nolimits}
\def\EIG{\mathop{\textrm{\rm EIG}}\nolimits}
\def\CIRCUIT{\mathop{\textrm{\rm CIRCUIT}}\nolimits}
\def\ZINTDISC{\ZZ\mathop{\textrm{\rm -INT-DISC}}\nolimits}
\def\SINTDISC{{S^1}\!\mathop{\textrm{\rm -ARC-DISC}}\nolimits}
\def\bfi{{\mathbf i}}
\def\bfv{{\mathbf v}}
\def\bfA{{\mathbf A}}
\def\bfone{{\mathbf1}}
\def\EE{{\mathds E}}

\def\ZZ{{\mathds Z}}
\def\PP{{\mathds P}}
\def\RR{{\mathds R}}
\def\CC{{\mathds C}}
\def\tG{{\widetilde G}} 
\def\cC{\mathcal C}

\newtheoremstyle{note}{4pt}{4pt}{\sl}{}{\bfseries}{.}{.5em}{}
\newtheoremstyle{introthms}{3pt}{3pt}{\itshape}{}{\bfseries}{.}{.5em}{\thmnote{#3}}
\newtheoremstyle{cases}{2pt}{2pt}{\rm}{}{\bfseries}{.}{.3em}{}

\theoremstyle{plain}
\newtheorem{theorem}              {Theorem}       [section]
\newtheorem{claim}      [theorem] {Claim}         
\newtheorem{lemma}      [theorem] {Lemma}         
\newtheorem{corollary}  [theorem] {Corollary}     
\newtheorem{fact}       [theorem] {Fact}          
       
\newtheorem{setup}      [theorem] {Setup} 

\theoremstyle{cases}
\newtheorem{case}                 {Case}  

\theoremstyle{note} 
\newtheorem{remark}     [theorem] {Remark}

\newtheorem{definition} [theorem] {Definition}       

\begin{document}

\title[Discrepancy and Eigenvalues]{Discrepancy and Eigenvalues of Cayley Graphs}

\dedicatory{Dedicated to the memory of Professor Miroslav Fiedler}

\author{Yoshiharu Kohayakawa}
\address{Instituto de Matem\'atica e Estat\'{\i}stica, Universidade de
  S\~ao Paulo, S\~ao Paulo, Brazil}  
\email{yoshi@ime.usp.br}

\author[Vojt\v{e}ch R\"{o}dl]{Vojt\v{e}ch R\"{o}dl}
\address{Department of Mathematics and Computer Science, 
Emory University, Atlanta, USA}
\email{rodl@mathcs.emory.edu}

\author{Mathias Schacht}
\address{Fachbereich Mathematik, Universit\"at Hamburg,
  Hamburg, Germany}
\email{schacht@math.uni-hamburg.de}

\thanks{The first author was supported by FAPESP (2013/03447-6,
  2013/07699-0), CNPq (459335/2014-6, 310974/2013-5) and Project
  MaCLinC/USP.  The second author was supported by NSF grant DMS 1301698.
  The third author was supported through the \emph{Heisenberg-Programme} of the DFG.  
  The collaboration of the first and third authors is supported
  by CAPES/DAAD PROBRAL project~430/15.}

\subjclass[2010]{Primary: 05C50. Secondary: 05C80}
\keywords{Eigenvalues, discrepancy, quasirandomness, Cayley graphs}
\date{\today}

\begin{abstract}
  We consider quasirandom properties for Cayley graphs of finite abelian
  groups.  We show that having uniform edge-distribution (i.e., small
  discrepancy) and having large eigenvalue gap are equivalent properties for
  such Cayley graphs, even if they are \emph{sparse}.  This positively answers
  a question of Chung and Graham [``Sparse quasi-random graphs'',
  Combinatorica \textbf{22} (2002), no.~2, 217--244] for the particular case
  of Cayley graphs of abelian groups, while in general the answer is negative.
\end{abstract}

\maketitle
\thispagestyle{empty}

\section{Introduction}
\label{sec:intro}
Professor Miroslav Fiedler discovered a very fruitful relationship
between connectivity properties of graphs and their spectra.  Among
other things, his works~\cites{fiedler73:_algeb, fiedler75} from the
1970s, together with other pioneering work~\cites{hall70:_r_dimen,
  donath72:_algor, donath73:_lower}, gave birth to what is now known
as \textit{spectral partitioning of graphs}.  Fiedler considered the
so called combinatorial Laplacian~$L(G)$ of graphs~$G$ and their
spectrum $0=\lambda_1\leq\lambda_2\leq\dots\leq\lambda_n$ ($n=|V(G)|$).
Generalizing the fact that~$G$ is connected if and only
if~$\lambda_2>0$, Fiedler named~$\lambda_2$ the \textit{algebraic
  connectivity} of~$G$ and went on to prove that~$\lambda_2$ is a
lower bound for the standard connectivity of~$G$ (unless~$G$ is the
complete graph).  Furthermore, he also considered partitioning the
vertex set of~$G$ by considering the coordinates of the eigenvector
belonging to~$\lambda_2$.  The algebraic connectivity of a graph is
now sometimes referred to as the \textit{Fiedler value} and the
associated eigenvector is referred to as the \textit{Fiedler vector}.
Alon~\cite{alon86:_eigen} and Sinclair and
Jerrum~\cite{sinclair89:_approx_markov} later proved that graphs with
small Fiedler value can be partitioned according to the Fiedler vector
in a direct way to produce a cut that is small in relative terms (that
is, in terms of the ratio of the number of cut edges to the number of
separated vertices).

While a small Fiedler value tells us that the graph in question may be
split along a ``small cut'', a large Fiedler value implies that the
graph is an \textit{expander}, that is, it has no cuts that are
``small''~\cites{alon85, tanner84:_explic_n}.  In this paper, we
investigate the relation between such ``edge-distribution properties''
and spectra, but focusing on the case of ``uniform
edge-distribution'', by which we mean the \textit{quasirandom} case,
in the sense of Chung, Graham and
Wilson~\cite{chung89:_quasi}.\footnote{Owing to this focus, spectral
  graph partitioning will not be discussed here; the interested reader
  is referred to, e.g., Spielman~\cite{spielman12:_spect} and Spielman
  and Teng~\cite{spielman07:_spect}.}  Since we shall be concerned
with Cayley graphs, which are regular graphs, for simplicity, we shall
work with adjacency matrices and \textit{not} with combinatorial
Laplacians.

Let an $n$-vertex graph~$G$ be given.  The \textit{eigenvalues} of~$G$
are simply the eigenvalues of the~$n$~by~$n$, $0$--$1$~adjacency
matrix of~$G$, with~$1$ indicating edges.
Let~$\lambda_k=\lambda_k(G)$ be the $k$th largest eigenvalue of~$G$,
in absolute value.  Recall that~$G$ is said to be ``quasirandom'' if
the edges of~$G$ are ``uniformly distributed'' (we postpone the
precise definition; see Definition~\ref{def:DISC_delta}).  A
fundamental result relating the~$\lambda_i$ to quasirandomness states
that \textsl{there is a large gap between~$\lambda_1$ and~$\lambda_k$
  $(k\geq2)$ \textit{if and only if} $G$~is quasirandom}.

The assertion above may be turned precise in different ways.  We are
interested in the form given by Chung, Graham, and
Wilson~\cite{chung89:_quasi}.  Recall that~\cite{chung89:_quasi} presents a
``theory of quasirandomness'' for graphs, exhibiting several, quite disparate
almost sure properties of graphs that are, quite surprisingly, equivalent in a
deterministic sense.  Earlier work in this direction is due to
Thomason~\cite{thomason87} (see also~\cite{thomason87:_random}), 
and also Alon~\cite{alon86:_eigen}, Alon and
Chung~\cite{alon88:_explic}, Frankl, R\"odl and Wilson~\cite{frankl88:_hadam},
and R\"odl~\cite{rodl86}.  One of the so-called ``quasirandom properties''
that is presented in~\cite{chung89:_quasi} is the ``eigenvalue gap'' 
between~$\lambda_1$ and~$\lambda_k$ ($k\geq2$).

Chung and Graham~\cite{ChGr02} set out to investigate the
extension of the results in~\cite{chung89:_quasi} to \textit{sparse graphs},
that is, graphs with vanishing edge-density.  As it turns out, a na\"\i ve
approach to such a project is doomed to fail, as the results
in~\cite{chung89:_quasi} \textit{do not} generalize to the ``sparse case'' in
the expected manner (for a thorough discussion on this point, the interested
reader is referred to~\cite{ChGr02} and also to~\cites{ACHKRS10,BL06,CFZ14,KohaRo01_rpI,KoRoSi03,KrSu03}).
In particular, having succeeded in proving that eigenvalue gap does imply
uniform distribution of edges in the sparse case, Chung and Graham asked whether
the converse also holds (see~\cite{ChGr02}*{p.~230}).  An affirmative answer to
this question would fully generalize the relationship between these two
concepts to the sparse case.

However, Krivelevich and Sudakov~\cite{KrSu03} showed that
the answer to the question posed by Chung and Graham is negative, by
constructing a suitable family of counterexamples. 
Here, our aim
is to show that \textit{the answer is positive if one considers Cayley graphs
  of finite abelian groups, regardless of the density of the graph.}
It is worth noting
that several explicit constructions of quasirandom graphs are indeed
Cayley graphs (see, e.g.,~\cite{thomason87:_random} and~\cite{KrSu03}*{Section~3}).

We use the following notation.  If~$G=(V,E)$ is a graph, we write~$e(G)$ for
the number of edges~$|E|$ in~$G$.  If~$U\subset V$ is a set of vertices
of~$G$, then~$G[U]$ denotes the subgraph of~$G$ induced by~$U$.  Furthermore,
if~$W\subset V$ is disjoint from~$U$, then we write~$G[U,W]$ for the bipartite
subgraph of~$G$ naturally induced by the pair~$(U,W)$.  We also sometimes
write~$E(U,W)=E_G(U,W)$ for the edge set of~$G[U,W]$.

If~$\delta>0$, we write~$x\sim_{\delta} y$ to mean that
\begin{equation*}
  \label{eq:def_sim_delta}
  (1-\delta)y\leq x\leq(1+\delta)y.
\end{equation*}
Moreover, sometimes it will be convenient to write~$O_1(\delta)$ 
for any term~$\beta$ that satisfies~$|\beta|\leq \delta$.
Observe that, clearly $x\sim_{\delta} y$ is equivalent to~$x=(1+O_1(\delta))y$.

\begin{definition}[$\DISC(\delta)$]
  \label{def:DISC_delta}
  Let~$0<\delta\leq1$ be given.  We say that an $n$-vertex graph~$G$
  $(n\geq2)$ satisfies property~$\DISC(\delta)$ if the following assertion
  holds: for all~$U\subset V(G)$ with~$|U|\geq\delta n$, we have
  \begin{equation*}
    \label{eq:def_DISC}
    e_G(U)=e(G[U])\sim_\delta e(G)\binom{|U|}{2}\left/\binom{n}{2}\right.\,.
  \end{equation*}
\end{definition}
The following concept of $\DISC_2$ is very much related to $\DISC$, as we shall see
next.

\begin{definition}[$\DISC_2(\delta')$]
  \label{def:DISC_2_delta}
  Let~$0<\delta'\leq1$ be given.  We say that an $n$-vertex graph~$G$ $(n\geq2)$
  satisfies property~$\DISC_2(\delta')$ if the following assertion
  holds: for all disjoint~$U$ and~$W\subset V(G)$ with~$|U|$,
  $|W|\geq\delta'n$, we have
  \begin{equation*}
    \label{eq:def_DISC_2}
    e_G(U,W)=e(G[U,W])\sim_{\delta'}e(G)|U||W|\left/\binom{n}{2}\right..
  \end{equation*}
\end{definition}

The following fact is very easy to prove and we omit its proof.

\begin{fact}
  \label{lem:DISC_implies_DISC_2}
  For any~$0<\delta'\leq1$, there is~$0<\delta=\delta(\delta')\leq1$ such that
  any graph that satisfies~$\DISC(\delta)$ must also
  satisfy~$\DISC_2(\delta')$.
\end{fact}

Given a graph~$G$, let~$\bfA=(a_{uv})_{u,v\in V(G)}$ be
the $0$--$1$~adjacency matrix of~$G$, with~$1$ denoting edges.  The
\textit{eigenvalues} of~$G$ are simply the eigenvalues of~$\bfA$.
Since~$\bfA$ is symmetric, its eigenvalues are real.  As usual, we adjust the
notation so that these eigenvalues are such that
\begin{equation}
  \label{eq:def_EIG_the_eigenvalues}
  \lambda_1\geq|\lambda_2|\geq\dots\geq|\lambda_n|
\end{equation}
(the fact that~$\lambda_1\geq0$ follows, for instance, from the fact
that the sum of the~$\lambda_i$ is equal to the trace of~$\bfA$, which
is~$0$).
 
\begin{definition}[$\EIG(\eps)$]
  \label{def:EIG_epsilon}
  Let~$0<\eps\leq1$ be given.  We say that an $n$-vertex graph~$G$ satisfies
  property~$\EIG(\eps)$ if the following holds.  Let~$\bar d=\bar
  d(G)=2e(G)/n$ be the average degree of~$G$, and
  let~$\lambda_1,\dots,\lambda_n$ be the eigenvalues of~$G$, with the notation
  adjusted in such a way that~\eqref{eq:def_EIG_the_eigenvalues} holds.
  Then
  \begin{enumerate}[label=\rmlabel]
  \item\label{it:eigi} $\lambda_1\sim_\eps \bar d$, 
  \item\label{it:eigii} $|\lambda_i|\leq\eps \bar d$ for all~$1<i\leq
    n$. 
  \end{enumerate}
\end{definition}

Finally, we define Cayley graphs.

\begin{definition}[Cayley graph~$G(\Gamma,A)$]
  \label{def:Cayley}
  Let~$\Gamma$ be an abelian group and let~$A\subset\Gamma\setminus\{0\}$
  be symmetric, that is, $A=-A$.  The \textit{Cayley graph}~$G=G(\Gamma,A)$ is
  defined to be the graph on~$\Gamma$, with two vertices~$\gamma$
  and~$\gamma'\in\Gamma$ adjacent in~$G$ if and only if~$\gamma'-\gamma\in A$.
\end{definition} 

We only consider finite graphs and finite abelian groups.
The main aim is to answer a question of Chung and Graham
from~\cite{ChGr02} in the positive for an interesting class of graphs.

\begin{theorem}
  \label{thm:main_positive} 
  For every~$\eps>0$, there exist~$\delta>0$ and~$n_0$ such that the following holds.
  Let~$G=G(\Gamma,A)$ be a Cayley graph for some 
  abelian group~$\Gamma$ with $n=|\Gamma|\geq n_0$ elements 
  and a symmetric
  set~$A=-A\subseteq\Gamma\setminus\{0\}$. 
  If~$G$ satisfies property~$\DISC(\delta)$,
  then~$G$ satisfies~$\EIG(\eps)$.
\end{theorem}

The proof of this theorem is given in Section~\ref{sec:pf_main}.  
We close this introduction with a few remarks concerning Theorem~\ref{thm:main_positive}.

We first observe that Theorem~\ref{thm:main_positive}, together with
the results of Chung and Graham~\cite{ChGr02}, imply that properties
$\DISC$ and $\EIG$ are equivalent for Cayley graphs.  More precisely,
by \textit{$\DISC$ implies $\EIG$ for Cayley graphs} we mean the
following: for every $\eps>0$ there is a $\delta=\delta(\eps)>0$ such
that, for any sequence of positive integers~$(n_k)_k$
with~$n_k\to\infty$ as~$k\to\infty$, and any sequence~$(G_k)_k$ 
of 
Cayley graphs with~$|V(G_k)|=n_k$,
we have that {\sl if all but finitely many graphs~$G_k$ satisfy
  $\DISC(\delta)$, then all but finitely many~$G_k$ satisfy
  $\EIG(\eps)$.}  Theorem~\ref{thm:main_positive} tells us that
$\DISC$ implies $\EIG$ for sequences of Cayley graphs.
In~\cite{ChGr02}*{Theorem~1} it is proved that $\EIG$ implies $\DISC$
in the same sense for sequences of arbitrary graphs with average
degree tending to infinity.  This establishes the equivalence of the
properties $\DISC$ and $\EIG$ for Cayley graphs with diverging average
degree. 

Secondly, we note that in general it is not true that 
$\DISC$ implies $\EIG$ for arbitrary sequences of graphs.
This was already pointed out by Krivelevich and Sudakov
in~\cite{KrSu03}.  For every $\eps>0$ and every $\delta>0$, they constructed
an infinite sequence of graphs that satisfy $\DISC(\delta)$ but fail to
satisfy \textrm{(\textit{i})} in the definition of~$\EIG(\eps)$ (see
Definition~\ref{def:EIG_epsilon}).

The following example is a different probabilistic construction: For
$p=p(n)\to0$ with~$pn\gg1$ as $n\to\infty$, consider the graph $G$
given by the union of the random graph $G(n,p)$ and a disjoint clique
of size $\alpha pn$ for some constant~$\alpha >0$. Such a graph $G$
has density $(1+o(1))p$ and for every fixed $\delta>0$ with high
probability it satisfies $\DISC(\delta)$.  However, $\alpha pn-1$ is
one of the eigenvalues of its adjacency matrix and, hence, $G$ fails
to satisfy~\ref{it:eigii} in the definition of~$\EIG(\eps)$
for any fixed $\eps\in(0,\alpha)$.

We also remark that in~\cite{ChGr02}, it is proved that, under some
additional conditions, $\DISC$ implies $\EIG$ for sequences of sparse graphs.
This additional assumption combined with $\DISC$ implies that almost every graph 
in the sequence contains the ``expected number'' of closed walks of length $\l$
for some even $\l\geq 4$. More precisely, for a sequence of graphs~$G_n$
with average degree~$\bar d_n$ we say it satisfies $\CIRCUIT_\l$ if 
the number of closed walks of length~$\l$ in~$G_n$ 
is $(1+o(1))(\bar d_n)^\l$. We remark that 
Theorem~\ref{thm:main_positive} is not a consequence of the result of Chung and Graham, since 
there exist sequences of Cayley graphs satisfying $\DISC$, and hence by Theorem~\ref{thm:main_positive} 
also $\EIG$, but fail to have $\CIRCUIT_\l$ for any fixed even
$\l\geq4$.
We next sketch the construction of such a sequence.

Let
\[
	p=p(n)=\frac{\log^2n}{n}
\] 
and consider the random cyclic Cayley graph $\cC_{n,p}=G(\ZZ/n\ZZ,A)$, where 
independently 
for every $a\in (\ZZ/n\ZZ)\setminus\{0\}$ both elements $a$ and $-a$ are included in $A$
with probability $p/2$. It follows from standard Chernoff-type estimates 
that asymptotically almost surely $\cC_{n,p}$ satisfies $\DISC$ and has average degree $\bar d_n=(1+o(1)pn$. 
Consequently, by Theorem~\ref{thm:main_positive} it also satisfies~$\EIG$. 

On the other hand, owing to the choice of $p$ we have 
\[
	pn^2\gg(pn)^\l
\]
for every fixed even $\l\geq 4$ and sufficiently large $n$. Hence, for every 
even $\l\geq 4$ in expectation 
the number of ``degenerated walks'' which only use one edge 
is $\gg(\bar d_n)^\l$. This implies that with positive probability~$\cC_{n,p}$ 
satisfies $\DISC$ and $\EIG$, but fails to satisfy $\CIRCUIT_\l$
for every even $\l\geq 4$. Using appropriate blowups of such 
graphs yields sequences of Cayley graphs with these properties for any density $p$
with \mbox{$\log^2 n/n \ll p\ll1$}.

Finally, we remark that very recently 
Conlon and Zhao~\cite{CZ} extended 
Theorem~\ref{thm:main_positive} for Cayley graphs for arbitrary 
(not necessarily abelian) finite groups.

\subsection*{Acknowledgements}
\label{sec:acknowdgements}

The proof of Theorem~\ref{thm:main_positive} presented here is
based on an idea of Tim Gowers~\cite{Gowers}.  The authors proved this result with a
longer combinatorial argument, which we include in the appendix.  
On learning about the result,
Tim Gowers suggested the alternative, elegant proof given below.
We are grateful to him for letting us include his proof here.

\section{Proof of the main result}
\label{sec:pf_main}

\subsection{Eigenvalues of Cayley graphs of abelian groups}
\label{sec:evalues_of_Cayley}
Theorem~\ref{thm:lovasz} below tells us how to compute the eigenvalues of
Cayley graphs of abelian groups (Theorem~\ref{thm:lovasz} follows from a more
general result due to Lov\'asz~\cite{lovasz75:_spect}; see
also~\cite{Lo07}*{Exercise~11.8} and~\cite{babai79:_spect_cayley}).

Before we state Theorem~\ref{thm:lovasz}, we recall some basic facts about
group characters (for more details see, e.g., Serre~\cite{Se77}).  
Let~$\Gamma$ be a finite abelian group.  In this
case, an \textit{irreducible character}~$\chi$ of~$\Gamma$ may be
viewed as a group homomorphism
$\chi\:\Gamma\to S^1$, i.e., $\chi(a+b)=\chi(a)\chi(b)$ for all $a$, $b\in \Gamma$, where~$S^1$ is the multiplicative group of complex
numbers of absolute value~$1$.
If~$\Gamma$ has order~$n$, then there are~$n$ irreducible  characters, say,
$\chi_1,\dots,\chi_n$, and these characters satisfy the following
\textit{orthogonality property}:
\begin{equation}
  \label{eq:ortho}
  \<\chi_i,\chi_j\>=\sum_{\gamma\in\Gamma}\chi_i(\gamma)\chi_j(\gamma)=0
\end{equation}
for all~$i\neq j$.  These facts and a simple computation suffice to prove the
following well known result, the short proof of which we include for
completeness.  We shall use the following notation: if~$X$ is a set, we also 
write~$X$ for the $\{0,1\}$-indicator function of~$X$, so that~$X(a)=1$
if~$a\in X$ and~$X(a)=0$ otherwise.

\begin{theorem}
  \label{thm:lovasz}
  Let~$G=G(\Gamma,A)$ be a Cayley graph for some finite abelian group~$\Gamma$ and a symmetric
  set~$A=-A\subseteq\Gamma\setminus\{0\}$.  For any character~$\chi\:\Gamma\to
  S^1$ of~$\Gamma$, put
  \begin{equation}
    \label{eq:lovasz}
    \lambda^{(\chi)}=\<A,\chi\>=\sum_{a\in A}\chi(a).
  \end{equation}
  Then the eigenvalues of~$G$ are the~$\lambda^{(\chi)}$, where~$\chi$
  runs over all~$n=|\Gamma|$ irreducible characters of~$\Gamma$.
\end{theorem}
\begin{proof}
  Let~$\chi\:\Gamma\to S^1$ be an irreducible character of~$\Gamma$. 
  Let~$\lambda^{(\chi)}$ be as defined in~\eqref{eq:lovasz}.
  Consider the vector $\bfv^{(\chi)}=(\chi(\gamma))_{\gamma\in\Gamma}^T$, with
  entries indexed by the elements of~$\Gamma=V(G)$.
  Let~$\bfA=(a_{\gamma\gamma'})_{\gamma,\gamma'\in\Gamma}$ be the
  adjacency matrix of~$G$.  
    
  Fix~$\gamma\in\Gamma$.  Observe that the
  $\gamma$-entry~$(\bfA\bfv^{(\chi)})_\gamma$ of the vector~$\bfA\bfv^{(\chi)}$
  is
  \begin{equation*}
    \label{eq:gamma-entry}
    (\bfA\bfv^{(\chi)})_\gamma=\sum_{a\in A}\chi(\gamma-a)=
    \sum_{a\in A}\chi(\gamma+a)
        =\Big(\sum_{a\in A}\chi(a)\Big)\chi(\gamma)
    =\lambda^{(\chi)}\chi(\gamma)\,,
  \end{equation*}
  and hence~$\bfA\bfv^{(\chi)}=\lambda^{(\chi)}\bfv^{(\chi)}$; that is,
  $\bfv^{(\chi)}$ is an eigenvector of~$\bfA$ with
  eigenvalue~$\lambda^{(\chi)}$.
  
  Let $\chi_j\:\Gamma\to S^1$ ($1\leq j\leq n$) be the irreducible characters
  of~$\Gamma$ and set~$\bfv_j=\bfv^{(\chi_j)}$ for all~$1\leq j\leq n$.
  By~\eqref{eq:ortho}, $\<\bfv_j,\bfv_{j'}\>=0$ if~$j\neq j'$.  Therefore,
  the~$\bfv_j$ ($1\leq j\leq n$) form an orthogonal basis of eigenvectors of
  the matrix~$\bfA$ and, hence, $\lambda^{(\chi_j)}$ ($j=1,\dots,n$) are
  indeed all the eigenvalues of~$G$.
\end{proof}

\begin{remark}
  \label{rem:triv_eigenvalue}
  The eigenvalue~$\lambda_1=d=|A|$ may be obtained from~\eqref{eq:lovasz} by
  letting~$\chi$ be the trivial character~$\chi(x)=1$ for all~$x\in\Gamma$.
\end{remark}

\subsection{The proof}
\label{sec:proof}
We shall prove that $\neg\EIG(\epsilon)\Rightarrow\neg\DISC(\delta)$.  By
Theorem~\ref{thm:lovasz} and Remark~\ref{rem:triv_eigenvalue}, our
assumption implies that there is a character~$\chi\not\equiv1$ such that
\begin{equation}
  \label{eq:assumption}
  |\lambda^{(\chi)}|=|\<A,\chi\>|\geq\epsilon|A|.  
\end{equation}
We shall fix this~$\chi$ and we shall use it to construct sets~$X$
and~$Y\subset V(G)$ that ``witness'' the fact that~$\neg\DISC(\delta)$ holds.

First we introduce some notation.  Let~$0\leq\chia(\gamma)<2\pi$ be defined
by~$\chi(\gamma)=\e^{\I\chia(\gamma)}$.  For~$\gamma\in\Gamma$, let
\begin{equation}
  \label{eq:c_def}
  c(\gamma)=\Re(\chi(\gamma))=\cos(\chia(\gamma))
\end{equation}
and 
\begin{equation}
  \label{eq:s_def}
  s(\gamma)=\Im(\chi(\gamma))=\sin(\chia(\gamma)).
\end{equation}
Applying the orthogonality relation~\eqref{eq:ortho} to~$\chi$ and the trivial
character~$\chi\equiv1$, denoted below by~$\bfone$, gives us that
\begin{equation}
  \label{eq:ortho.1}
  0=\<\bfone,\chi\>=\sum_{\gamma\in\Gamma}\e^{\I\chia(\gamma)}
  =\sum_{\gamma\in\Gamma}\(c(\gamma)+\I s(\gamma)\).
\end{equation}
Consequently, 
\begin{equation}
  \label{eq:triv_sum}
  \sum_{\gamma\in\Gamma}c(\gamma)
  =\sum_{\gamma\in\Gamma}s(\gamma)=0.
\end{equation}

Given two functions~$f$ and~$g\:\Gamma\to\CC$, let~$f*g\:\Gamma\to\CC$ be
their \textit{convolution}, given by
\begin{equation}
  \label{eq:conv_def}
  (f*g)(\alpha)=\sum_{\gamma\in\Gamma}f(\alpha-\gamma)g(\gamma).
\end{equation}
In what follows, we let~$m$ be the cardinality of the image of~$\chi$:
\begin{equation}
  \label{eq:m_def}
  m=|\{\chi(\gamma)\colond\gamma\in\Gamma\}|.
\end{equation}
Since~$\chi\not\equiv1$, we have~$m>1$.  We shall need the following
fact.

\begin{lemma}
  \label{lem:ident}
  We have
  \begin{enumerate}[label=\rmlabel]
  \item \label{it:identi} 
    \begin{equation}
      \label{eq:sum_c^2.1}
      \sum_{\gamma\in\Gamma}c^2(\gamma)=
      \begin{cases}
        n   &\text{if~$m=2$}\\
        n/2   &\text{if~$m>2$;}
      \end{cases}
    \end{equation}
  \item \label{it:identii}
    \begin{align}
      \<A,{1\over2}(1+c)*\frac{1}{2}(1+c)\>
      &={1\over4}n|A|+{1\over4}\<A,c*c\>\label{eq:ident}\\
      &=
      \begin{cases}
        {1\over4}n|A|+{1\over4}n\<A,c\>  &\text{if~$m=2$}\\
        {1\over4}n|A|+{1\over8}n\<A,c\>  &\text{if~$m>2$}.
      \end{cases}
      \label{eq:ident.2}
    \end{align}
  \end{enumerate}
\end{lemma}

We postpone the proof of Lemma~\ref{lem:ident} to
Section~\ref{sec:proof_of_fact}, and proceed to prove our main theorem.
Let~$-X$ and~$Y\subset\Gamma$ be generated at random as follows: we
include~$\gamma\in\Gamma$ in~$-X$ with probability~$p(\gamma)=(1+c(\gamma))/2$
and we include~$\gamma\in\Gamma$ in~$Y$ with the same probability~$p(\gamma)$.
with all these events independent.

By~\eqref{eq:triv_sum} we have~$\sum_{\gamma\in\Gamma}p(\gamma)=n/2$.
Therefore, by a Chernoff type inequality (see, e.g., Alon and
Spencer~\cite{AlSp08}*{Theorem~A.1.4}), we have
\begin{equation}
  \label{eq:|X|}
  \PP\(|X|=\({1\over2}+o(1)\)n\)=1-o(1)
\end{equation}
and
\begin{equation}
  \label{eq:|Y|}
  \PP\(|Y|=\({1\over2}+o(1)\)n\)=1-o(1).
\end{equation}
In view of Lemma~\ref{lem:ident}~\ref{it:identi},
we have
\[
\sum_{\gamma\in\Gamma}p(-\gamma)p(\gamma)
=\sum_{\gamma\in\Gamma}p^2(\gamma)
=\frac{1}{4}\sum_{\gamma\in\Gamma}(1+c(\gamma))^2
\overset{\eqref{eq:triv_sum}}{=}\frac{1}{4}n+\frac{1}{4}\sum_{\gamma\in\Gamma}c(\gamma)^2
=\frac{3}{8}n
\] 
if~$m>2$
and~$\sum_{\gamma\in\Gamma}p(-\gamma)p(\gamma)=n/2$ if~$m=2$.  Consequently,
if~$m>2$, we have
\begin{equation}
  \label{eq:|XcapY|}
  \PP\(|X\cap Y|=\({3\over8}+o(1)\)n\)=1-o(1)
\end{equation}
and hence, in view of~\eqref{eq:|X|} and~\eqref{eq:|Y|}, we have 
\begin{equation}
  \label{eq:|XcupY|}
  \PP\(|X\cup Y|=\({5\over8}+o(1)\)n\)=1-o(1).
\end{equation}
Similarly, if~$m=2$, we have
\begin{equation}
  \label{eq:|XcapY|.2}
  \PP\(|X\cap Y|=\({1\over2}+o(1)\)n\)=1-o(1)
\end{equation}
and
\begin{equation}
  \label{eq:|XcupY|.2}
  \PP\(|X\cup Y|=\({1\over2}+o(1)\)n\)=1-o(1).
\end{equation}

On the other hand, in view of our assumption~\eqref{eq:assumption}
and $A=-A$ we have
\begin{equation}
  \label{eq:assumption_c}
  \epsilon|A|\leq|\<A,\chi\>|=|\<A,c\>|.
\end{equation}
Recall that~$p(\gamma)=(1+c(\gamma))/2$ is the probability that we
include~$\gamma$ in~$-X$ and in~$Y$.  By the linearity of the expectation and
the independence, we have\footnote{In~\eqref{eq:expectation}, we write~$(-X)$
  for the characteristic function of the set~$-X=\{-x\:x\in X\}$.}
\begin{align}
  \EE(\<A,(-X)*Y\>)
  &=\EE\Big(\sum_{a\in A}\sum_{\gamma\in\Gamma}
  (-X)(a-\gamma)Y(\gamma)\Big)\nonumber\\
  &=\sum_{a\in A}\sum_{\gamma\in\Gamma}
  \EE\((-X)(a-\gamma)\)\EE\(Y(\gamma)\)
  =\sum_{a\in A}\sum_{\gamma\in\Gamma}
  p(a-\gamma)p(\gamma)\nonumber\\
  &=\<A,{1\over2}(1+c)*{1\over2}(1+c)\>.
  \label{eq:expectation}
\end{align}
By Lemma~\ref{lem:ident}~\ref{it:identii}, we thus have
\begin{equation}
  \label{eq:expectation.2}
  \left|\EE(\<A,(-X)*Y\>)-{1\over4}n|A|\right|
  \geq{1\over8}n|\<A,c\>|\geq{1\over8}\epsilon n|A|.
\end{equation}
On the other hand, 
\begin{multline}
  \label{eq:from_E2e(X,Y)}
  \<A,(-X)*Y\>=\sum_{a\in A}\sum_{\gamma\in\Gamma}(-X)(a-\gamma)Y(\gamma)
  =\sum_{a\in A}\sum_{\gamma\in\Gamma}X(-a+\gamma)Y(\gamma)
  =e(X,Y),
\end{multline}
with the edges in~$X\cap Y$ counted twice.  Since~$0\leq e(X,Y)\leq n|A|$, the
random variable
\begin{equation}
  \label{eq:Z_def}
  \eta=\eta(X,Y)=\<A,(-X)*Y\>-{1\over4}n|A|=e(X,Y)-{1\over4}n|A|
\end{equation}
satisfies
\begin{equation}
  \label{eq:Z_bds}
  -{1\over4}n|A|\leq\eta\leq{3\over4}n|A|.
\end{equation}
Let~$q$ be the probability that~$|\eta|\leq\epsilon n|A|/16$.  Then,
by~\eqref{eq:expectation.2} and~\eqref{eq:Z_bds},
\begin{equation}
  \label{eq:p_bound}
  {1\over8}\epsilon n|A|\leq|\EE(\eta)|
  \leq\EE(|\eta|)\leq{1\over16}\epsilon n|A|q+{3\over4}n|A|(1-q),
\end{equation}
and, consequently,
\begin{equation}
  \label{eq:p_bounded.2}
  \PP\(|\eta|\leq{1\over16}\epsilon n|A|\)
  =q\leq{1-\epsilon/6\over1-\epsilon/12}\leq1-{1\over12}\epsilon.
\end{equation}
First consider the case in which~$m>2$.  Putting
together~\eqref{eq:|X|}--\eqref{eq:|XcupY|} and~\eqref{eq:p_bounded.2} we see
that there are sets~$X$ and~$Y\subset\Gamma$ for which we have
\begin{equation}
  \label{eq:|X|_correct}
  |X|=\({1\over2}+o(1)\)n,\qquad
  |Y|=\({1\over2}+o(1)\)n,
\end{equation}
\begin{equation}
  \label{eq:|XcapY|_correct}
  |X\cap Y|=\({3\over8}+o(1)\)n,\qquad
  |X\cup Y|=\({5\over8}+o(1)\)n,
\end{equation}
and
\begin{equation}
  \label{eq:e(X,Y)_deviant}
  \left|e(X,Y)-{1\over4}n|A|\right|
  \geq{1\over16}\epsilon n|A|.
\end{equation}
Fix such sets~$X$ and~$Y$.  Suppose that none of the sets~$X\setminus Y$,
$Y\setminus X$, $X\cup Y$, and~$X\cap Y$ violates~$\DISC(\delta)$.
Then for sufficiently large $n$ we have
\begin{equation}
  \label{eq:e(X-Y)}
  \left|e(X\setminus Y)-{1\over128}n|A|\right|<{2\over128}\delta n|A|,
  \qquad
  \left|e(Y\setminus X)-{1\over128}n|A|\right|<{2\over128}\delta n|A|,
\end{equation}
and
\begin{equation}
  \label{eq:e(XcapY)}
  \left|e(X\cap Y)-{9\over128}n|A|\right|<{10\over128}\delta n|A|,
  \qquad
  \left|e(Y\cup X)-{25\over128}n|A|\right|<{26\over128}\delta n|A|.
\end{equation}
Since
\begin{equation}
  \label{eq:e(X,Y)_decomp}
  e(X,Y)=e(X\cup Y)-e(X\setminus Y)-e(Y\setminus X)+e(X\cup Y),
\end{equation}
we infer that 
\begin{equation}
  \label{eq:e(X,Y)_deviation}
  \left|e(X,Y)-{32\over128}n|A|\right|<{40\over128}\delta n|A|,
\end{equation}
which contradicts~\eqref{eq:e(X,Y)_deviant} if~$\delta\leq\epsilon/5$.  The
proof for the case~$m>2$ is finished.

The case~$m=2$ is similar.  Putting together~\eqref{eq:|X|}, \eqref{eq:|Y|},
\eqref{eq:|XcapY|.2}, \eqref{eq:|XcupY|.2}, and~\eqref{eq:p_bounded.2} we see
that there are sets~$X$ and~$Y\subset\Gamma$ for which we have
\begin{equation}
  \label{eq:|X|_correct.2}
  |X|=\({1\over2}+o(1)\)n,\qquad
  |Y|=\({1\over2}+o(1)\)n,
\end{equation}
\begin{equation}
  \label{eq:|XcapY|_correct.2}
  |X\cap Y|=\({1\over2}+o(1)\)n,\qquad
  |X\cup Y|=\({1\over2}+o(1)\)n,
\end{equation}
and, moreover, with~$X$ and~$Y$ satisfying~\eqref{eq:e(X,Y)_deviant}.  Fix
such sets~$X$ and~$Y$.  Note that, then,
\begin{equation}
  \label{eq:e(X-Y).2}
  e(X\setminus Y)=o(n|A|)
  \quad\text{ and }\quad
  e(Y\setminus X)=o(n|A|).
\end{equation}
Suppose that neither~$X\cup Y$ nor~$X\cap Y$ violates~$\DISC(\delta)$.
Then for sufficiently large $n$ we have 
\begin{equation}
  \label{eq:e(XcapY).2}
  \left|e(X\cap Y)-{1\over8}n|A|\right|<{2\over8}\delta n|A|
  \quad\text{ and }\quad
  \left|e(Y\cup X)-{1\over8}n|A|\right|<{2\over8}\delta n|A|.
\end{equation}
Using~\eqref{eq:e(X,Y)_decomp} again, we infer that
\begin{equation}
  \label{eq:e(X,Y)_deviation.2}
  \left|e(X,Y)-{1\over4}n|A|\right|<{5\over8}\delta n|A|,
\end{equation}
which contradicts~\eqref{eq:e(X,Y)_deviant} if~$\delta\leq\epsilon/10$,
completing the proof in the case~$m=2$.

\subsection{Proof of Lemma~\ref{lem:ident}}
\label{sec:proof_of_fact}
We start with the following fact (Fact~\ref{fact:idents}~\ref{it:fidentsi} below is
simply Lemma~\ref{lem:ident}~\ref{it:identi}).

\begin{fact}
  \label{fact:idents}
  We have
  \begin{enumerate}[label=\rmlabel]
  \item\label{it:fidentsi} 
    \begin{equation}
      \label{eq:sum_c^2}
      \sum_{\gamma\in\Gamma}c^2(\gamma)=
      \begin{cases}
        n   &\text{if~$m=2$}\\
        n/2   &\text{if~$m>2$;}
      \end{cases}
    \end{equation}

  \item\label{it:fidentsii}
    \begin{equation}
      \label{eq:sum_sc}
      \sum_{\gamma\in\Gamma}s(\gamma)c(\gamma)=0;
    \end{equation}
  \item\label{it:fidentsiii} for any~$a\in\Gamma$
    \begin{equation}
      \label{eq:c*c(a)}
      (c*c)(a)=
      \begin{cases}
        nc(a)  &\text{if~$m=2$}\\
        (n/2)c(a)  &\text{if~$m>2$.}
      \end{cases}
    \end{equation}
  \end{enumerate}
\end{fact}
\begin{proof}
  \ref{it:fidentsi} We start by observing that
  \begin{equation}
    \label{eq:sum_cos}
    \sum_{0\leq\ell<m}\cos{4\pi\ell\over m}
    =
    \begin{cases}
      2   &\text{if~$m=2$}\\
      0   &\text{if~$m>2$}.\\
    \end{cases}
  \end{equation}
  Indeed, if~$m>2$, then the sum in~\eqref{eq:sum_cos} is
  \begin{equation}
    \label{sum_cos>2}
    \Re\sum_{0\leq\ell<m}\e^{4\pi\ell\I/m}
    =\Re{1-\e^{4\pi\I}
      \over1-\e^{4\pi\I/m}}
    =0.
  \end{equation}
  If~$m=2$, then the sum in~\eqref{eq:sum_cos} is easily seen to be~$2$.  We
  now observe that
  \begin{equation}
    \label{eq:sum_c^2.2}
    \sum_{\gamma\in\Gamma}c^2(\gamma)
    ={n\over m}\sum_{0\leq\ell<m}\cos^2\(2\pi\ell\over m\)
    ={n\over2m}\sum_{0\leq\ell<m}\(1+\cos{4\pi\ell\over m}\).
  \end{equation}
  It now suffices to recall~\eqref{eq:sum_cos} to deduce~\eqref{eq:sum_c^2};
  assertion~\ref{it:fidentsi} is therefore proved.  
  
  Now we prove~\ref{it:fidentsii}.
  Note that
  \begin{equation}
    \label{eq:sum_sin}
    \sum_{0\leq\ell<m}\sin{4\pi\ell\over m}=0.
  \end{equation}
  Therefore,
  \begin{equation}
    \label{eq:ii_pf}
    \sum_{\gamma\in\Gamma}s(\gamma)c(\gamma)
    ={n\over m}\sum_{0\leq\ell<m}
    \sin\(2\pi\ell\over m\)\cos\(2\pi\ell\over m\)
    ={n\over2m}\sum_{0\leq\ell<m}\sin{4\pi\ell\over m}=0,
  \end{equation}
  as required.  
  
  For the proof of~\ref{it:fidentsiii}, we start by noticing that 
  \begin{align*}
    c(a-\gamma)&=\cos(\chia(a-\gamma))
    =\cos\(\chia(a)-\chia(\gamma)\)\\
    &=\cos\chia(a)\cos\chia(\gamma)+\sin\chia(a)\sin\chia(\gamma)
    =c(a)c(\gamma)+s(a)s(\gamma).
  \end{align*}
  Therefore, 
  \begin{align*}
    (c*c)(a)&=\sum_{\gamma\in\Gamma}c(a-\gamma)c(\gamma)
    =\sum_{\gamma\in\Gamma}\(c(a)c(\gamma)+s(a)s(\gamma)\)c(\gamma)\\
    &=\sum_{\gamma\in\Gamma}\(c(a)c^2(\gamma)+s(a)s(\gamma)c(\gamma)\)
    =c(a)\sum_{\gamma\in\Gamma}c^2(\gamma)
    +s(a)\sum_{\gamma\in\Gamma}s(\gamma)c(\gamma).
  \end{align*}
  Eq.~\eqref{eq:c*c(a)}~follows from~\eqref{eq:sum_c^2} and~\eqref{eq:sum_sc}
  and~\ref{it:fidentsiii} is proved.
\end{proof}

\begin{proof}[Proof of Lemma~\ref{lem:ident}]
  Lemma~\ref{lem:ident}~\ref{it:identi} has already been proved.  We now turn
  to~\ref{it:identii}.  The left-hand side of~\eqref{eq:ident} is
  \begin{align}
    &{1\over4}\sum_{a\in A}\sum_{\gamma\in\Gamma}
    \((1+c)(a-\gamma)\)\((1+c)(\gamma)\)\nonumber\\
    &\qquad\qquad\qquad={1\over4}\sum_{a\in A}\sum_{\gamma\in\Gamma}
    \(1+c(a-\gamma)\)\(1+c(\gamma)\)\nonumber\\
    &\qquad\qquad\qquad={1\over4}n|A|
    +{1\over4}\sum_{a\in A}\sum_{\gamma\in\Gamma}
    \(c(a-\gamma)+c(\gamma)\)
    +{1\over4}\sum_{a\in A}\sum_{\gamma\in\Gamma}
    c(a-\gamma)c(\gamma)\nonumber\\
    &\qquad\qquad\qquad={1\over4}n|A|
    +{1\over4}\sum_{a\in A}\sum_{\gamma\in\Gamma}
    c(a-\gamma)c(\gamma)\nonumber\\
    &\qquad\qquad\qquad={1\over4}n|A|
    +{1\over4}\<A,c*c\>,\label{eq:ident.3}
  \end{align}
  which verifies~\eqref{eq:ident}.  Clearly,
  Fact~\ref{fact:idents}~\ref{it:fidentsiii} and~\eqref{eq:ident.3}
  imply~\eqref{eq:ident.2}.  
\end{proof}

\appendix
\section*{Appendix A}
\renewcommand{\thesection}{A}
\subsection{Combinatorial proof of Theorem~\ref{thm:main_positive}}
\label{sec:proof_of_main}
We include our original proof of Theorem~\ref{thm:main_positive} here. Let a
constant~$\eps>0$ be given.  The aim is to find some $\delta>0$ for which
property~$\DISC(\delta)$ implies~$\EIG(\eps)$ for any Cayley
graph~$G=G(\Gamma,A)$.  Let us once and for all fix an abelian group~$\Gamma$
and a symmetric set~$A\subseteq\Gamma\setminus\{0\}$.  In what follows, we
write~$G$ for the Cayley graph~$G(\Gamma,A)$.  We shall always write~$n$ for
the number of vertices in~$G$, i.e., $n=|\Gamma|=|V(G)|$.  We also
let~$|A|=\alpha n$.

Clearly, our graph~$G$ is $|A|$-regular.  Therefore, the density of
the graph~$G$ is
\begin{equation}\label{eq:density_G}
  e(G)\left/\binom{n}{2}\right.=\frac{|A|}{n-1}\,.
\end{equation}
Moreover,
as is well known,
condition~\textrm{(\textit{i})} of Definition~\ref{def:EIG_epsilon} is automatically
fulfilled.  We should therefore consider condition~(\textit{ii}) of that
definition. Because of Theorem~\ref{thm:lovasz}, our task is to estimate
the~$\lambda^{(\chi)}$ given in~\eqref{eq:lovasz}.  More precisely, we have to
show that if~$\chi\not\equiv1$, then
\begin{equation}
  \label{eq:ultimate_EIG_bd}
  |\lambda^{(\chi)}|=\Big|\sum_{a\in A}\chi(a)\Big|\leq\eps|A|.
\end{equation}
Thus, let~$\chi\:\Gamma\to S^1$ be a fixed, non-constant irreducible character
of~$\Gamma$.  We shall estimate~$\lambda^{(\chi)}$ in two different ways,
according to the cardinality of~$\im\chi=\{\chi(\gamma)\:\gamma\in\Gamma\}$.
In what follows, we always write~$m$ for~$|\im\chi|$. 
We also use the bijection $\e^{\theta\bfi}$, 
mapping every~$\theta$ in $\RR/2\pi\RR$ to $\e^{\theta\bfi}$ in $S^1$.
We define 
$$
\chia\colond\Gamma\to\RR/2\pi\RR
$$
to be the homomorphism such that for every $\gamma\in\Gamma$
\begin{align*}
\chia(\gamma)=\arg\(\chi(\gamma)\)\quad \text{and}\quad
\chi(\gamma)=\e^{\chia(\gamma)\bfi}\,.
\end{align*}

Furthermore, we let~$\Omega\:\ZZ/m\ZZ\to \RR/2\pi\RR$ be the homomorphism~
$$
\Omega(s)=\frac{2\pi}{m}s\quad\text{for}\quad s\in\ZZ/m\ZZ\,.
$$ 
We also have a homomorphism~$\rho\:\Gamma\to\ZZ/m\ZZ$ for
which~$\chia=\Omega\rho$ holds, so that
\[
	\chia(\gamma)=\frac{2\pi\rho(\gamma)}{m}
\]
for every~$\gamma$ in~$\Gamma$. 
Summarising the above, from now on we will work with the following setup.

\begin{setup}\label{setup:1}
  Let $G=G(\Gamma,A)$ be the Cayley graph given by the abelian group~$\Gamma$
  and the symmetric set $A=-A\subseteq\Gamma\setminus\{0\}$.  The graph~$G$ is
  of order~$n=|\Gamma|$, every vertex has degree $|A|=\alpha n$, and the density of the
  graph is $|A|/(n-1)$.

  Fix an irreducible character $\chi\not\equiv 1$, set $m=|\im\chi|$,
  and let $\chia$, $\Omega$, and $\rho$ (depending on $\chi$) be group
  homomorphisms such that the following diagram commutes:
  $$
  \xymatrix{
    \Gamma \ar[d]_{\rho} \ar[drr]^{\chia} \ar[rr]^{\chi} &&
    {S^1\subset\CC}\\
    {\ZZ/m\ZZ} \ar[rr]_{\Omega} && 
    {\RR/2\pi\RR} \ar[u]_{\e^{\theta\bfi}} 
  }
  $$
\end{setup}

As mentioned above 
we consider two cases for the proof of
Theorem~\ref{thm:main_positive}. 
In the first case~$m$ will be small. The following lemma will handle that case.

\begin{lemma}\label{lem:m_small}
  For every $\delta'>0$ there is an $n_0\geq 0$ such that if
  $|\Gamma|=n\geq n_0$, $m\leq 1/\delta'$, and $G=G(\Gamma,A)$ 
  satisfies $\DISC_2(\delta')$, then 
  $$
  |\lambda^{(\chi)}|\leq 2\delta'|A|\,.
  $$
\end{lemma}

For the other case ($m$ large), we shall need 
three auxiliary lemmas to verify~\eqref{eq:ultimate_EIG_bd}.  
The proofs of these three lemmas, as well as the proof of
Lemma~\ref{lem:m_small}, are given in 
Sections~\ref{sec:proof_of_m_small}--\ref{sec:ZINTDISC_to_EIG}.
 We start with two definitions.

\begin{definition}[$\ZINTDISC(\rho;\eta,\sigma)$]
  \label{def:Z-INT-DISC}
  For positive reals $\eta$ and $\sigma$, we say that~$A$
  satisfies $\ZINTDISC(\rho;\eta,\sigma)$
  if for all integers~$0\leq D_1<D_2\leq\lfloor m/2\rfloor+1$ such
  that~$D_2-D_1\geq\eta m$ we have
  \begin{equation}
    \label{eq:Z-INT-DISC}
    \left|A\cap\rho^{-1}\big([D_1,D_2)\big)\right|
      \sim_\sigma\frac{D_2-D_1}{m}|A|.
  \end{equation}
\end{definition}
Roughly speaking, a set~$A$ satisfies $\ZINTDISC$ if its
image under $\rho$ intersects ``large'' intervals uniformly.
Next we define a very similar property for~$A$ with respect to 
$\chia$ and intervals in~$\RR/2\pi\RR$.
\begin{definition}[$\SINTDISC(\chia;\eta,\sigma)$]
  \label{def:S-INT-DISC}
  For positive reals $\eta$ and $\sigma$, we say that $A$
  satisfies $\SINTDISC(\chia;\eta,\sigma)$
  if for all reals~$0\leq \theta_1<\theta_2\leq\pi$ such
  that~$\theta_2-\theta_1\geq2\pi\eta$ we have
  \begin{equation}
    \label{eq:S-INT-DISC}
    \left|A\cap\chia^{-1}\big([\theta_1,\theta_2]\big)\right|
      \sim_\sigma\frac{\theta_2-\theta_1}{2\pi}|A|.
  \end{equation}
\end{definition}

Basically, the next three lemmas give the following implications
for large~$m$:
$$
\DISC\Longrightarrow\ZINTDISC\Longrightarrow
\SINTDISC\Longrightarrow\EIG\,.
$$
These lemmas are stated under the assumptions of 
Setup~\ref{setup:1}; in particular, we recall that~$\rho$ and $\chia$
depend on the fixed, non-constant character $\chi$.
\begin{lemma}
  \label{lem:auxZ}
  For all positive reals~$\eta$ and $\sigma$, there
  are~$\delta=\delta(\eta,\sigma)>0$ and $n_0\geq 0$ such that
  if $|\Gamma|=n\geq n_0$, $m> 1/\delta$, and~$G=G(\Gamma,A)$ 
  satisfies~$\DISC(\delta)$, then~$A$ satisfies
  $\ZINTDISC(\rho;\eta,\sigma)$. 
\end{lemma}

\begin{lemma}
  \label{lem:translation}
  For all positive reals~$\eta\leq1$ and $\sigma\leq1$ 
  such that~$m\eta\sigma\geq3$, the following holds.  If~$A$
  satisfies property $\ZINTDISC(\rho;\eta/2,\sigma/3)$, then~$A$ satisfies
  property $\SINTDISC(\chia;\eta,\sigma)$.
\end{lemma}

\begin{lemma}
  \label{lem:auxS}
  For every real~$\eps>0$, there are reals~$\eta=\eta(\eps)>0$ and 
  $\sigma=\sigma(\eps)>0$ for which the following holds.  
  If~$A$ satisfies~$\SINTDISC(\chia;\eta,\sigma)$,
  then
  \begin{equation}
    \label{eq:auxS}
    |\lambda^{(\chi)}|=\Big|\sum_{a\in A}\chi(a)\Big|
    \leq\eps|A|.
  \end{equation}
\end{lemma}

We now assume Lemmas~\ref{lem:m_small}, \ref{lem:auxZ}, \ref{lem:translation},
and~\ref{lem:auxS} and give the proof of Theorem~\ref{thm:main_positive}.  (We
present the proofs of those auxiliary results in
Section~\ref{sec:auxiliaries}.)

\begin{proof}[Proof of Theorem~\ref{thm:main_positive}]
  Let~$\epsilon>0$ be given.  We apply Lemma~\ref{lem:auxS}, which yields the
  positive constants $\eta=\eta(\eps)$ and $\sigma=\sigma(\eps)$.  Then
  Lemma~\ref{lem:auxZ} gives
  $\delta_{\ref{lem:auxZ}}=\delta_{\ref{lem:auxZ}}(\eta/2,\sigma/3)$.  We set
  $$
  \delta'=\min\left\{\delta_{\ref{lem:auxZ}},
    \frac{\eta\sigma}{3},\frac{\eps}{2}\right\}\,.
  $$
  We now choose $\delta$ promised by Theorem~\ref{thm:main_positive} to be
  $$
  \delta=\min\{\delta_{\ref{lem:DISC_implies_DISC_2}}(\delta'),
  \delta_{\ref{lem:auxZ}}\}\,,
  $$
  where $\delta_{\ref{lem:DISC_implies_DISC_2}}(\delta')$ is given by
  Fact~\ref{lem:DISC_implies_DISC_2}.  Finally, let~$n_0$ be as
  large as required by Lemmas~\ref{lem:m_small} and~\ref{lem:auxZ}.  We claim
  that this choice for~$\delta$ and~$n_0$ will do, and proceed to check
  this claim.

  Suppose~$\DISC(\delta)$ holds for some Cayley graph $G=G(\Gamma,A)$ with
  $|\Gamma|\geq n_0$ and let $\chi\not\equiv 1$ be given (the notation here
  follows the notation set out in Setup~\ref{setup:1}). We consider two cases.
  
  Suppose first that~$m\leq 1/\delta'$.  Fact~\ref{lem:DISC_implies_DISC_2}
  tells us that $\DISC_2(\delta')$ holds for~$G$.  Since~$m\leq 1/\delta'$,
  Lemma~\ref{lem:m_small} tells us that $|\lambda^{(\chi)}|\leq\eps|A|$ by the
  choice of~$\delta'\leq\eps/2$.  For the other case, namely, $m>1/\delta'$,
  we first observe that $\DISC(\delta_{\ref{lem:auxZ}})$ holds since
  $\delta\leq\delta_{\ref{lem:auxZ}}$ and that $m>1/\delta'\geq
  1/\delta_{\ref{lem:auxZ}}$. Moreover, the choice
  of~$\delta'\leq\eta\sigma/3$ yields $m\eta\sigma>\eta\sigma/\delta'\geq 3$,
  making Lemma~\ref{lem:translation} applicable.  Our claim is now a
  consequence of the following implications coming from
  Lemmas~\ref{lem:auxZ}--\ref{lem:auxS}:
  \[ 
  	\DISC(\delta_{\ref{lem:auxZ}})
    \Longrightarrow\ZINTDISC\(\rho;\frac{\eta}{2},\frac{\sigma}{3}\)
    \Longrightarrow\SINTDISC(\chia;\eta,\sigma)\Longrightarrow
    |\lambda^{(\chi)}|\leq\eps|A|\,,
  \]
  and hence Theorem~\ref{thm:main_positive} is proved.
\end{proof}

\subsection{Proof of the auxiliary lemmas}
\label{sec:auxiliaries}
In this section we prove the auxiliary lemmas
Lemmas~\ref{lem:m_small}, \ref{lem:auxZ}, \ref{lem:translation},
and~\ref{lem:auxS}
\subsubsection{An auxiliary weighted graph}
\label{sec:weighted_graph}
The homomorphism~$\rho$ (see Setup~\ref{setup:1} for details), 
defines a weighted graph~$\tG$ on
$\ZZ/m\ZZ$ in a natural way. The symmetry of this graph will be
useful in the proofs of Lemmas~\ref{lem:m_small} and~\ref{lem:auxZ}.

\begin{definition}
  \label{def:weighted_graph}
  We let (under the assumptions of
  Setup~\ref{setup:1})~$\tG=\tG(\rho)=(\ZZ/m\ZZ,w)$ be the \textit{weighted
    graph} on~$\ZZ/m\ZZ$, with weights assigned to the edges and vertices,
  with the \textit{weight function}
  \begin{equation*}
    \label{eq:weighted_graph}
    w\:\binom{\ZZ/m\ZZ}{2}\cup\ZZ/m\ZZ\to\ZZ 
  \end{equation*}
  given by
  \begin{equation}
    \label{eq:def_w_diff}
    w(\{r,s\})=e(G[\rho^{-1}(r),\rho^{-1}(s)]),
  \end{equation}
  for all distinct~$r$ and~$s\in\ZZ/m\ZZ$, and
  \begin{equation}
    \label{eq:def_w_eq}
    w(r)=e(G[\rho^{-1}(r)]),
  \end{equation}
  for all~$r\in\ZZ/m\ZZ$.
\end{definition}

For convenience, if~$X$ and~$Y\subset\ZZ/m\ZZ$ are two disjoint sets, we put
\begin{equation}
  \label{eq:w_between_sets}
  w(X,Y)=\sum\Big\{w(\{x,y\})\:(x,y)\in X\times Y\Big\}
  =e\(\rho^{-1}(X),\rho^{-1}(Y)\)\,.  
\end{equation}
In Lemma~\ref{lem:weighted_graph} below, we make the definition of~$\tG$ more
concrete, computing the values in~\eqref{eq:def_w_diff}
and~\eqref{eq:def_w_eq}.  Let us observe that Lemma~\ref{lem:weighted_graph}
shows that the weighted graph~$\tG$ has a ``cyclic'' structure, that is, the
cyclic permutation $\tau\:s\mapsto s+1$ is an ``automorphism'' of~$\tG$.  

\begin{lemma}
  \label{lem:weighted_graph}
  For all distinct~$r$ and~$s\in\ZZ/m\ZZ$, we have 
  \begin{equation}
    \label{eq:weighted_lemma_diff}
    w(\{r,s\})=e(G[\rho^{-1}(r),\rho^{-1}(s)])
    =\frac{n}{m}\big|A\cap\rho^{-1}(r-s)\big|,
  \end{equation}
  and for all~$r\in\ZZ/m\ZZ$ we have
  \begin{equation}
    \label{eq:weighted_lemma_eq}
    w(r)=e(G[\rho^{-1}(r)])
    =\frac{n}{2m}\big|A\cap\rho^{-1}(0)\big|.
  \end{equation}
\end{lemma}
\begin{proof}
  Let $r$ and $s$ be
  arbitrary, not necessarily distinct members of $\ZZ/m\ZZ$. 
  For every~$\gamma$
  in $\rho^{-1}(s)$, consider the neighbourhood $N_r(\gamma)$ 
  of~$\gamma$ in~$G$ restricted to $\rho^{-1}(r)$. 
  It is easy to see that 
  $$
  N_r(\gamma)=\{A\cap\rho^{-1}(r-s)\}+\gamma
  $$
  for every $\gamma\in\rho^{-1}(s)$. 
  Since $|\rho^{-1}(s)|=n/m$, this implies, for $s\not=r$, that
  $$
  e(G[\rho^{-1}(r),\rho^{-1}(s)])=|\rho^{-1}(s)|\cdot|A\cap\rho^{-1}(r-s)|
  =\frac{n}{m}|A\cap\rho^{-1}(r-s)|\,,
  $$
  and therefore~\eqref{eq:weighted_lemma_diff} holds.
  Similarly,~\eqref{eq:weighted_lemma_eq} follows from the case $r=s$.
\end{proof}

\subsubsection{The small $m$ case} 
\label{sec:proof_of_m_small}
The proof given in this section is fairly simple. It is based
on~\eqref{eq:weighted_lemma_diff} combined with an
application of $\DISC_2$.

\begin{proof}[Proof of Lemma~\ref{lem:m_small}]
  Let $\delta'>0$ be given and let~$n_0$ large enough such that $n\sim_{\delta'/2}(n-1)$ for
  every~$n\geq n_0$.

  Now assume $G=G(\Gamma,A)$ with $|\Gamma|\geq n_0$ satisfying
  $\DISC_2(\delta')$ is given. Using~$1/m\geq\delta'$, 
  we deduce from~$\DISC_2(\delta')$ and~\eqref{eq:weighted_lemma_diff}
  that for all~$r\in\ZZ/m\ZZ$
  $$
  |A\cap\rho^{-1}(r)|=\frac{m}{n}e\(G[\rho^{-1}(r),\rho^{-1}(0)]\)
  \sim_{\delta'}\frac{n |A|}{(n-1)m}\,,
  $$
  and hence, by the choice of~$n_0$,
  \begin{equation}
    \label{eq:local_even_distribution}
    |A\cap\rho^{-1}(r)|\leq(1+2\delta')\frac{|A|}{m}\,.
  \end{equation}
  We then set $\omega=\e^{2\pi\bfi/m}$ and 
  use~\eqref{eq:local_even_distribution} to infer
  $$
  |\lambda^{(\chi)}|=\Big|\sum_{a\in A}\chi(a)\Big|=
  \Big|\sum_{r=0}^{m-1}\(|A\cap\rho^{-1}(r)|\cdot\omega^r\)\Big|
  \leq\frac{|A|}{m}\(\Big|\sum_{r=0}^{m-1}\omega^r\Big|+2\delta'm\)\,,
  $$
  which yields $|\lambda^{(\chi)}|\leq 2\delta'|A|$, because
  $$
  \sum_{r=0}^{m-1}\omega^r=0\,.
  $$
\end{proof}

\subsubsection{$\DISC$ implies $\ZINTDISC$ for~$m$ large}
\label{sec:proof_of_auxZ}
The aim of this section is to verify $\ZINTDISC(\rho;\eta,\sigma)$ for a graph
that satisfies $\DISC(\delta)$ for sufficiently small~$\delta$.  We therefore
want to link properties of the edge-distribution of~$G$ with the quantities
$$
|A\cap\rho^{-1}(I)|=\sum_{f\in I}|A\cap\rho^{-1}(f)|\,,
$$
where~$I$ is a sufficiently large interval in $\ZZ/m\ZZ$. A first step
towards this goal is the following lemma.  

\begin{lemma}
  \label{lem:weight_between_two_intervals}
  Let~$\ell$, $s$, and~$t$ be integers, and suppose that
  $$
  0\leq s<s+\ell\leq t<t+\ell\leq m/2\,.  
  $$
  Then for~$d_1=t-s-\ell$ and~$d_2=t-s+\ell$  
  \begin{equation} 
    \label{eq:weight_between_two_intervals}
    \begin{split}
      \frac{m}{n}w\big([s,s+\ell),[t,t+\ell)\big)
      &=\sum\left\{\big|A\cap\rho^{-1}(f)\big|(f-d_1)\:d_1<f<t-s\right\}\\
      &\qquad+\sum\left\{\big|A\cap\rho^{-1}(f)\big|(d_2-f)\:t-s\leq
        f<d_2\right\}\,.  
    \end{split}
  \end{equation}
\end{lemma}
We later may control the left-hand side
of~\eqref{eq:weight_between_two_intervals} by $\DISC$ (or, more
precisely, by $\DISC_2$). On the other hand, we may interpret the 
right-hand side as a ``weighted version'' of
$|A\cap\rho^{-1}([d_1,d_2])|$    
where the ``multiplicity'' for each $f$ in $[d_1,d_2]$ is given by
a piecewise linear function depending on $d_1$ and $d_2$ (see   
Figure~\ref{fig:1}).  

\begin{figure}[ht] 
  \begin{picture}(0,0)%
    \includegraphics{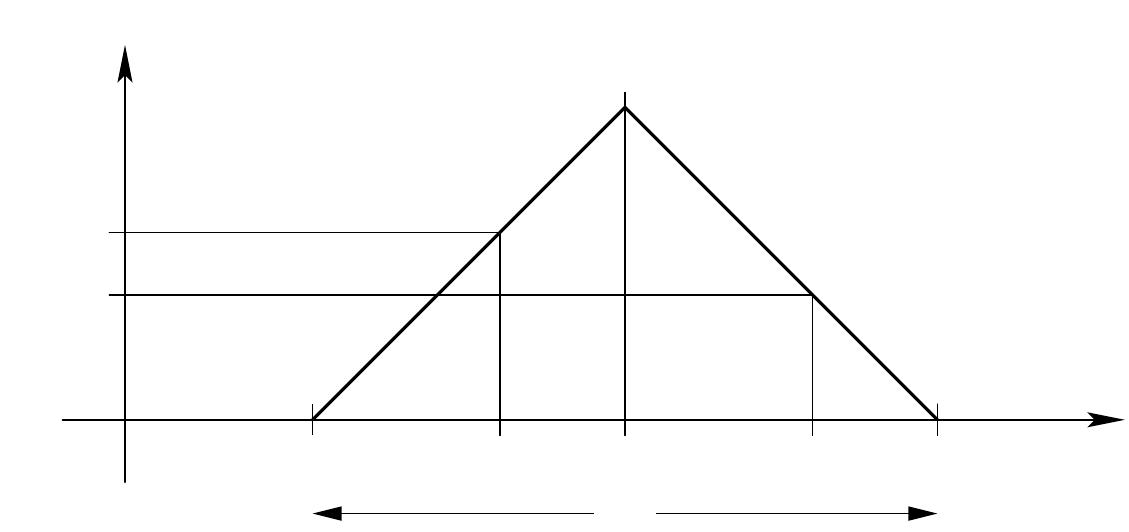}%
  \end{picture}%
  \setlength{\unitlength}{3947sp}%
  \begingroup\makeatletter\ifx\SetFigFont\undefined%
  \gdef\SetFigFont#1#2#3#4#5{%
    \reset@font\fontsize{#1}{#2pt}%
    \fontfamily{#3}\fontseries{#4}\fontshape{#5}%
    \selectfont}%
  \fi\endgroup%
  \begin{picture}(5412,2518)(1,-2996) 
    \put(80,-586){\makebox(0,0)[lb]{\smash{\SetFigFont{9}{10.8}{\familydefault}{\mddefault}{\updefault}{``multiplicities''}}}}
    \put(2930,-2956){\makebox(0,0)[lb]{\smash{\SetFigFont{9}{10.8}{\familydefault}{\mddefault}{\updefault}{$2\ell$}}}}
    \put(60,-1600){\makebox(0,0)[lb]{\smash{\SetFigFont{9}{10.8}{\familydefault}{\mddefault}{\updefault}{$f-d_1$}}}}
    \put(4,-1900){\makebox(0,0)[lb]{\smash{\SetFigFont{9}{10.8}{\familydefault}{\mddefault}{\updefault}{$d_2-f'$}}}}
    \put(2345,-2686){\makebox(0,0)[lb]{\smash{\SetFigFont{9}{10.8}{\familydefault}{\mddefault}{\updefault}{$f$}}}}
    \put(3850,-2686){\makebox(0,0)[lb]{\smash{\SetFigFont{9}{10.8}{\familydefault}{\mddefault}{\updefault}{$f'$}}}}
    \put(2850,-2686){\makebox(0,0)[lb]{\smash{\SetFigFont{9}{10.8}{\familydefault}{\mddefault}{\updefault}{$t-s$}}}}
    \put(1440,-2686){\makebox(0,0)[lb]{\smash{\SetFigFont{9}{10.8}{\familydefault}{\mddefault}{\updefault}{$d_1$}}}}
    \put(4440,-2686){\makebox(0,0)[lb]{\smash{\SetFigFont{9}{10.8}{\familydefault}{\mddefault}{\updefault}{$d_2$}}}}
  \end{picture}
  \caption{Distribution of ``multiplicities''}\label{fig:1}   
\end{figure}
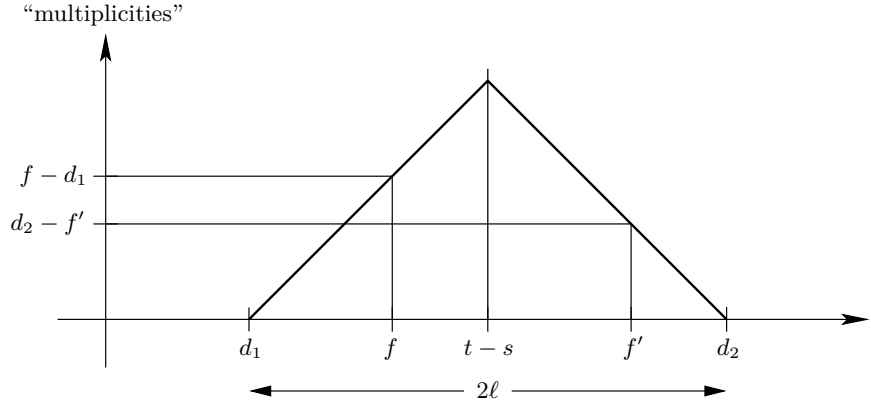
  
\begin{proof}[Proof of Lemma~\ref{lem:weight_between_two_intervals}]
  We have 
  \begin{equation} 
    \label{eq:pf_weight_between_two_intervals}  
    w\big([s,s+\ell),[t,t+\ell)\big)
    =\sum\Big\{w(e)\:e\in E_{\tG}\big([s,s+\ell),[t,t+\ell)\big)\Big\}.
  \end{equation}
  The integers~$f$ that arise as differences~$t'-s'$ with~$t'\in[t,t+\ell)$
  and~$s'\in[s,s+\ell)$ are in the interval 
  \begin{equation}
    \label{eq:f_interval}
    d_1=t-s-\ell< f< t-s+\ell=d_2.
  \end{equation}
  Intuitively speaking, these are the ``lengths'' of the edges
  in~$E_{\tG}\big([s,s+\ell),[t,t+\ell)\big)$.  A moment's thought
  (see Figure~\ref{fig:3}) shows that
  assertions~\ref{it:AI} and~\ref{it:AII} given below hold.

  \begin{enumerate}[label=\RMlabel]
  \item\label{it:AI} If~$f$ is in the interval
      $d_1=t-s-\ell< f<t-s$,  
    then~$f-d_1$ edges in the graph~$E_{\tG}\big([s,s+\ell),[t,t+\ell)\big)$ have
    length~$f$. 
  \item\label{it:AII} If~$f$ is in the interval 
      $t-s\leq f<d_2=t-s+\ell$,  
    then~$d_2-f$ edges in the graph~$E_{\tG}\big([s,s+\ell),[t,t+\ell)\big)$ have
    length~$f$.
  \end{enumerate}

  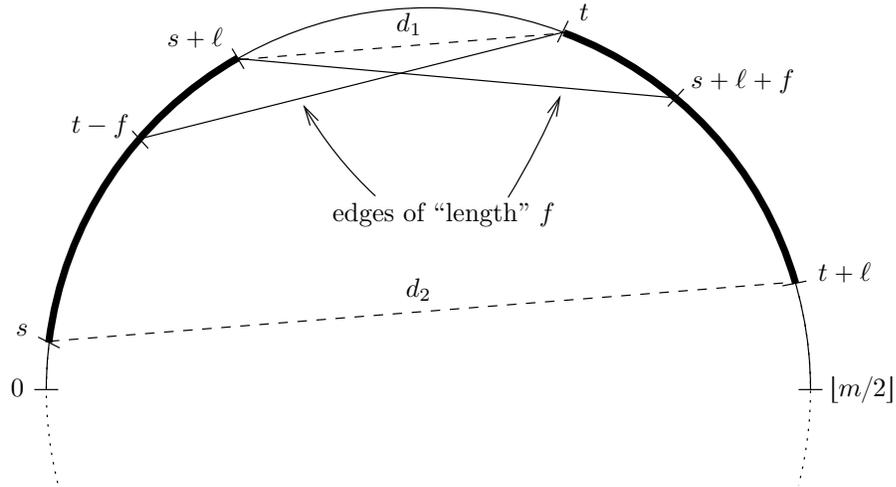
\begin{figure}[ht]
    \begin{picture}(0,0)%
      \includegraphics{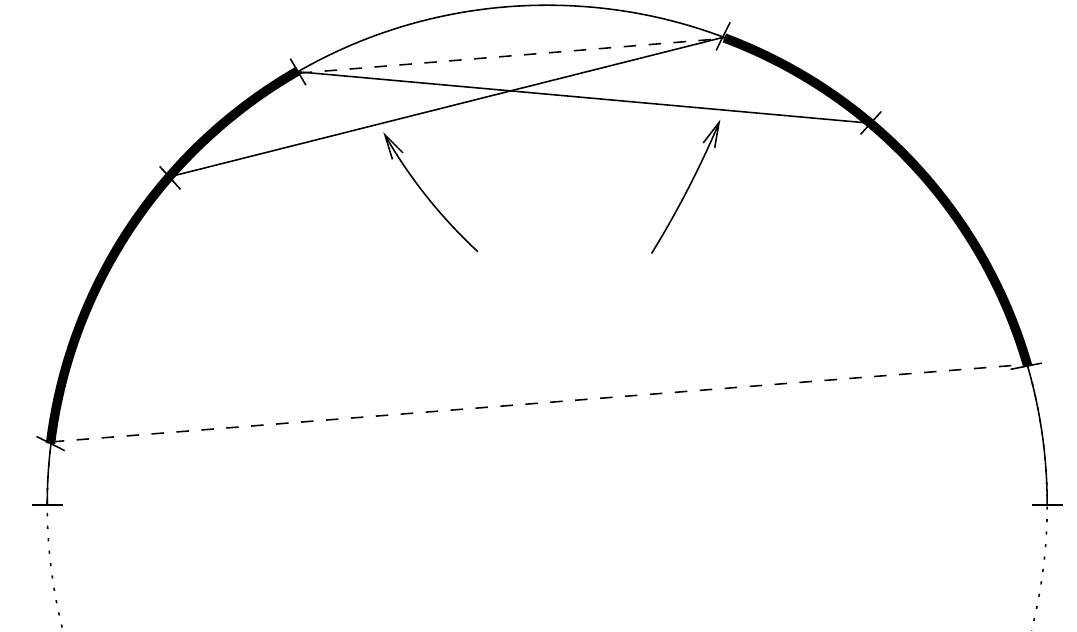}%
    \end{picture}%
    \setlength{\unitlength}{3947sp}%
    \begingroup\makeatletter\ifx\SetFigFont\undefined%
    \gdef\SetFigFont#1#2#3#4#5{%
      \reset@font\fontsize{#1}{#2pt}%
      \fontfamily{#3}\fontseries{#4}\fontshape{#5}%
      \selectfont}%
    \fi\endgroup%
    \begin{picture}(5160,3017)(375,-3069)
      \put(380,-2500){\makebox(0,0)[lb]{\smash{\SetFigFont{10}{12.0}{\familydefault}{\mddefault}{\updefault}{$0$}}}}
      \put(410,-2120){\makebox(0,0)[lb]{\smash{\SetFigFont{10}{12.0}{\familydefault}{\mddefault}{\updefault}{$s$}}}}
      \put(5450,-1780){\makebox(0,0)[lb]{\smash{\SetFigFont{10}{12.0}{\familydefault}{\mddefault}{\updefault}{$t+\ell$}}}}
      \put(1360,-300){\makebox(0,0)[lb]{\smash{\SetFigFont{10}{12.0}{\familydefault}{\mddefault}{\updefault}{$s+\ell$}}}}
      \put(3950,-140){\makebox(0,0)[lb]{\smash{\SetFigFont{10}{12.0}{\familydefault}{\mddefault}{\updefault}{$t$}}}}
      \put(2800,-210){\makebox(0,0)[lb]{\smash{\SetFigFont{10}{12.0}{\familydefault}{\mddefault}{\updefault}{$d_1$}}}}
      \put(2860,-1880){\makebox(0,0)[lb]{\smash{\SetFigFont{10}{12.0}{\familydefault}{\mddefault}{\updefault}{$d_2$}}}}
      \put(5520,-2500){\makebox(0,0)[lb]{\smash{\SetFigFont{10}{12.0}{\familydefault}{\mddefault}{\updefault}{$\lfloor m/2\rfloor$}}}}
      \put(760,-849){\makebox(0,0)[lb]{\smash{\SetFigFont{10}{12.0}{\familydefault}{\mddefault}{\updefault}{$t-f$}}}}
      \put(4650,-565){\makebox(0,0)[lb]{\smash{\SetFigFont{10}{12.0}{\familydefault}{\mddefault}{\updefault}{$s+\ell+f$}}}}
      \put(2400,-1401){\makebox(0,0)[lb]{\smash{\SetFigFont{10}{12.0}{\familydefault}{\mddefault}{\updefault}{edges of ``length'' $f$}}}}
    \end{picture}
    \caption{Edges in $\tG$ of ``length'' $d_1<f<t-s$ appear between $[t-f,s+\ell)$
      and $[t,s+\ell+f)$}\label{fig:3} 
  \end{figure}

  In other words, the lengths~$f$ in the interval~$(d_1,t-s)$
  occur~$f-d_1$, times and the lengths~$f$ in the
  interval~$[t-s,d_2)$ occur~$d_2-f$ times in the sum
  in~\eqref{eq:pf_weight_between_two_intervals}.  Each occurrence of~$f$
  contributes to~\eqref{eq:pf_weight_between_two_intervals} a weight of
  \begin{equation} 
    \label{eq:each_contribution}
    \frac{n}{m}|A\cap\rho^{-1}(f)|
  \end{equation} 
  (see~\eqref{eq:weighted_lemma_diff}).  Therefore, putting~\ref{it:AI},~\ref{it:AII},
  and~\eqref{eq:each_contribution} together,
  identity~\eqref{eq:weight_between_two_intervals} follows.  
\end{proof}

The next step towards verifying $\ZINTDISC$ is to dispose of the
``multiplicities'' of type $d_1-f$ and $d_2-f$
in~\eqref{eq:weight_between_two_intervals}.  For this we use
Lemma~\ref{lem:weight_between_two_intervals} and we
compare the quantities~$w\big([s-d,s+\ell),[t-d,t+\ell)\big)$ with
$w\big([s,s+\ell),[t,t+\ell)\big)$.
By~\eqref{eq:weight_between_two_intervals} these two terms (appropriately
scaled) correspond to two ``weighted versions'' of
$|A\cap\rho^{-1}([d_1',d_2'])|$ and $|A\cap\rho^{-1}([d_1,d_2])|$, for some
appropriate $d_1'$, $d_2'$, $d_1$, and $d_2$ depending on $d$, $s$, $t$, and
$\ell$.  As it turns out, the difference between these two ``weighted
versions'' yields a ``weighted version'' of~$|A\cap\rho^{-1}([d_1',d_2'])|$
with constant multiplicity~$d$ for the main part of the interval
$[d_1',d_2']$, i.e., in between $d_1$ and $d_2$ (see Figure~\ref{fig:2}). This
way we derive a useful estimate for $|A\cap\rho^{-1}([d_1,d_2])|$.
   
\begin{figure}[ht]
  \begin{picture}(0,0)%
    \includegraphics{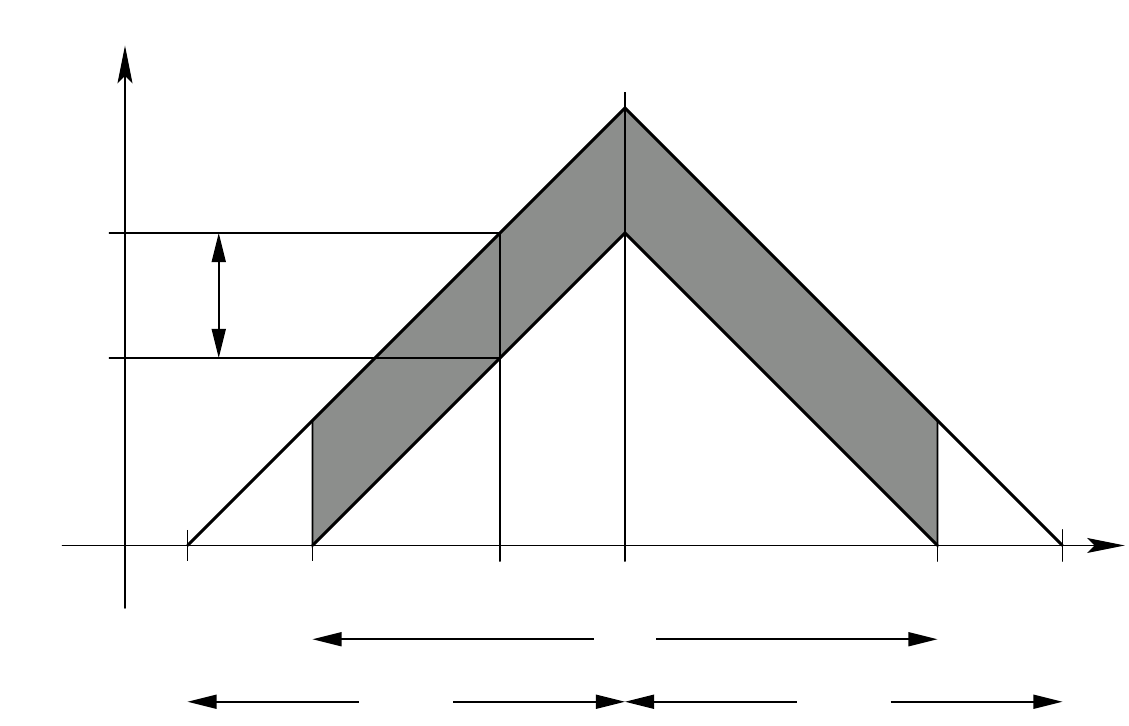}%
  \end{picture}%
  \setlength{\unitlength}{3947sp}%
  \begingroup\makeatletter\ifx\SetFigFont\undefined%
  \gdef\SetFigFont#1#2#3#4#5{%
    \reset@font\fontsize{#1}{#2pt}%
    \fontfamily{#3}\fontseries{#4}\fontshape{#5}%
    \selectfont}%
  \fi\endgroup%
  \begin{picture}(5412,3448)(1,-3326)
    \put(4440,-2686){\makebox(0,0)[lb]{\smash{\SetFigFont{9}{10.8}{\familydefault}{\mddefault}{\updefault}{$d_2$}}}}
    \put(1440,-2686){\makebox(0,0)[lb]{\smash{\SetFigFont{9}{10.8}{\familydefault}{\mddefault}{\updefault}{$d_1$}}}}
    \put(2850,-2686){\makebox(0,0)[lb]{\smash{\SetFigFont{9}{10.8}{\familydefault}{\mddefault}{\updefault}{$t-s$}}}}
    \put(846,-2686){\makebox(0,0)[lb]{\smash{\SetFigFont{9}{10.8}{\familydefault}{\mddefault}{\updefault}{$d_1'$}}}}
    \put(5045,-2686){\makebox(0,0)[lb]{\smash{\SetFigFont{9}{10.8}{\familydefault}{\mddefault}{\updefault}{$d_2'$}}}}
    \put(2345,-2686){\makebox(0,0)[lb]{\smash{\SetFigFont{9}{10.8}{\familydefault}{\mddefault}{\updefault}{$f$}}}}
    \put(  60,-1600){\makebox(0,0)[lb]{\smash{\SetFigFont{9}{10.8}{\familydefault}{\mddefault}{\updefault}{$f-d_1$}}}}
    \put(  60,-1008){\makebox(0,0)[lb]{\smash{\SetFigFont{9}{10.8}{\familydefault}{\mddefault}{\updefault}{$f-d_1'$}}}}
    \put(2930,-2956){\makebox(0,0)[lb]{\smash{\SetFigFont{9}{10.8}{\familydefault}{\mddefault}{\updefault}{$2\ell$}}}}
    \put(1790,-3243){\makebox(0,0)[lb]{\smash{\SetFigFont{9}{10.8}{\familydefault}{\mddefault}{\updefault}{$\ell+d$}}}}
    \put(3895,-3243){\makebox(0,0)[lb]{\smash{\SetFigFont{9}{10.8}{\familydefault}{\mddefault}{\updefault}{$\ell+d$}}}}
    \put( 80, 14){\makebox(0,0)[lb]{\smash{\SetFigFont{9}{10.8}{\familydefault}{\mddefault}{\updefault}{``multiplicities''}}}}
    \put(1115,-1320){\makebox(0,0)[lb]{\smash{\SetFigFont{9}{10.8}{\familydefault}{\mddefault}{\updefault}{$d$}}}}
  \end{picture}
  \caption{Difference between ``multiplicities''}\label{fig:2}   
\end{figure}
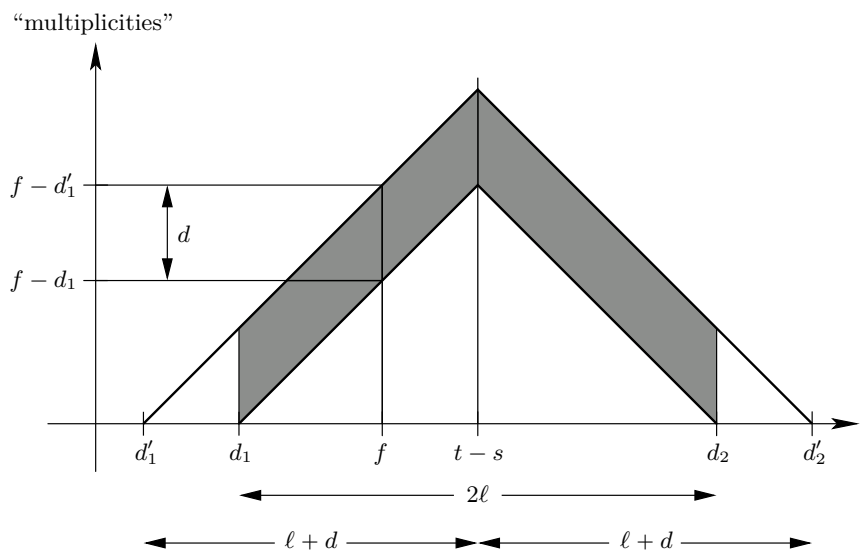
  
\begin{lemma}
  \label{lem:distribution_of_A}
  Let~$d$, $\ell$, $s$, and~$t$ be positive integers, $\delta'$ a 
  real number such that $0<\delta'\leq 1$, and suppose that
  \begin{enumerate}[label=\rmlabel]
  \item $0\leq s-d<s+\ell\leq t-d<t+\ell\leq  
    m/2$,
  \item $\DISC_2(\delta')$ holds 
    for~$G=G(\Gamma,A)$,  
  \item $\ell-d\geq\delta' m$, and  
  \item $n\sim_{\delta'/2} (n-1)$.    
  \end{enumerate} 
  Then, for $d_1=t-s-\ell$ and $d_2=t-s+\ell$, we have     
  \begin{equation} 
    \label{eq:distribution_of_A} 
    \left||A\cap\rho^{-1}\big([d_1,d_2)\big)|-2\ell\frac{|A|}{m}\right|
    \leq\frac{|A|}{m}\(d+\frac{2\delta'}{d}\((\ell+d)^2+\ell^2\)\)\,. 
  \end{equation} 
\end{lemma}   
\begin{proof}
  Let $d_1'=(t-d)-s-\ell=d_1-d$ and $d_2'=t+\ell-(s-d)=d_2+d$.  Applying
  Lemma~\ref{lem:weight_between_two_intervals}
  to~$w\big([s-d,s+\ell),[t-d,t+\ell)\big)$, we get that
  \begin{multline*}
    \label{eq:weight_between_outer_intervals}
    \frac{m}{n}w\big([s-d,s+\ell),[t-d,t+\ell)\big)\\
    =\sum\left\{\big|A\cap\rho^{-1}(f)\big|(f-d_1')\:
      d_1'< f<t-s\right\}\\
    +\sum\left\{\big|A\cap\rho^{-1}(f)\big|(d_2'-f)\:
      t-s\leq f<d_2'\right\}.  
  \end{multline*}
  We then apply Lemma~\ref{lem:weight_between_two_intervals} again, now
  to~$w\big([s,s+\ell),[t,t+\ell)\big)$, and observe that
  \begin{equation}
    \label{eq:diff_of_weights}
    \begin{split}
      &\frac{m}{n}\Big(w\big([s-d,s+\ell),[t-d,t+\ell)\big)
      -w\big([s,s+\ell),[t,t+\ell)\big)\Big)\\
      &\qquad=d\sum\left\{\big|A\cap\rho^{-1}(f)\big|\:d_1\leq f<d_2\right\}\\
      &\qquad\qquad+\sum\left\{\big|A\cap\rho^{-1}(f)\big|(f-d_1')\:
        d_1'< f< d_1\right\}\\
      &\qquad\qquad+\sum\left\{\big|A\cap\rho^{-1}(f)\big|(d_2'-f)\:
        d_2\leq f<d_2'\right\}.
    \end{split}
  \end{equation}
  The ``main term'' on the right-hand side of~\eqref{eq:diff_of_weights} will
  turn out to be
  \begin{equation}
    \label{eq:main_term}
    d\sum\left\{\big|A\cap\rho^{-1}(f)\big|\:d_1\leq f<d_2\right\}
    =d\big|A\cap\rho^{-1}\big([d_1,d_2)\big)\big|. 
  \end{equation}
  
  We now use~$\DISC_2(\delta')$ to estimate the left-hand side
  of~\eqref{eq:diff_of_weights}. By the definition of~$w$ (see
  Definition~\ref{def:weighted_graph}), 
  we have
  \begin{equation*}
    \label{eq:weight_as_no_edges}
    w([s,s+\ell),[t,t+\ell))=e\(G[\rho^{-1}([s,s+\ell)),\rho^{-1}([t,t+\ell))]\).
  \end{equation*}
  Therefore, by~$\DISC_2(\delta')$, using that
  \begin{equation*}
    \label{eq:size_lw_bds}
    \big|\rho^{-1}\big([s,s+\ell)\big)\big|
    =\big|\rho^{-1}\big([t,t+\ell)\big)\big|
    =\frac{n}{m}\ell\geq\delta' n,
  \end{equation*}
  we have that
  \begin{align}
    \label{eq:first_w_estimate}
      w\big([s,s+\ell),[t,t+\ell)\big)
      &=e\(G[\rho^{-1}([s,s+\ell)\big),\rho^{-1}\big([t,t+\ell))]\)\\
      &\sim_{\delta'}\frac{|A|}{n-1}\Big|\rho^{-1}\big([s,s+\ell)\big)\Big|
      \Big|\rho^{-1}\big([t,t+\ell)\big)\Big|
      =\frac{|A|}{n-1}\(\frac{n}{m}\ell\)^2\,.\nonumber
  \end{align}
  Similarly, we have that 
  \begin{equation}
    \label{eq:second_w_estimate}
    w\big([s-d,s+\ell),[t-d,t+\ell)\big)
    \sim_{\delta'}\frac{|A|}{n-1}\(\frac{n}{m}(\ell+d)\)^2.
  \end{equation}
  From~\eqref{eq:first_w_estimate},~\eqref{eq:second_w_estimate}, 
  and~\textrm{(\textit{iv})} we deduce
  that the left-hand side of~\eqref{eq:diff_of_weights} satisfies
  \begin{multline}\label{eq:LHS_estimate}
    \frac{m}{n}\Big(w\big([s-d,s+\ell),[t-d,t+\ell)\big)
    -w\big([s,s+\ell),[t,t+\ell)\big)\Big)\\
    =(1+O_1(2\delta'))\frac{|A|}{m}(\ell+d)^2
    -(1+O_1(2\delta'))\frac{|A|}{m}\ell^2 \\
    =\frac{|A|}{m}
    \(2\ell d+d^2+O_1(2\delta')\((\ell+d)^2+\ell^2\)\).
  \end{multline}
  Therefore, replacing the left-hand side
  of~\eqref{eq:diff_of_weights} by~\eqref{eq:LHS_estimate} and 
  using~\eqref{eq:main_term} immediately yields
  \begin{equation}
    \label{eq:still_id}
    \begin{split}
      &d\big|A\cap\rho^{-1}\big([d_1,d_2)\big)\big|\\
      &\qquad\qquad+\sum\left\{\big|A\cap\rho^{-1}(f)\big|(f-d_1')\:
        d_1'< f<d_1\right\}\\
      &\qquad\qquad+\sum\left\{\big|A\cap\rho^{-1}(f)\big|(d_2'-f)\:
        d_2\leq f<d_2'\right\}\\
      &\qquad\qquad\qquad\qquad
      =\frac{|A|}{m}\(2\ell d+d^2+O_1(2\delta')\((\ell+d)^2+\ell^2\)\). 
    \end{split}
  \end{equation}
  Clearly, \eqref{eq:still_id}~implies that 
  \begin{equation}
    \label{eq:ultimate_up_bd}
    d\big|A\cap\rho^{-1}\big([d_1,d_2)\big)\big|
    \leq\frac{|A|}{m}\(2\ell d+d^2+O_1(2\delta')\((\ell+d)^2+\ell^2\)\). 
  \end{equation}
  Moreover, we observe that
  \begin{align*}
    d\big|A\cap\rho^{-1}\big([d_1,d_2)\big)\big|&\geq
    d\big|A\cap\rho^{-1}\big([d_1+d,d_2-d)\big)\big|\\
    &\qquad+\sum\left\{\big|A\cap\rho^{-1}(f)\big|(f-d_1)\:
      d_1\leq f<d_1+d\right\}\\
    &\qquad+\sum\left\{\big|A\cap\rho^{-1}(f)\big|(d_2-f)\:
      d_2-d\leq f<d_2\right\}\\
    &=\frac{m}{n}\Big(w\big([s,s+\ell),[t,t+\ell)\big)
    -w\big([s+d,s+\ell),[t+d,t+\ell)\big)\Big)\,,
  \end{align*}
  where the last identity follows from
  Lemma~\ref{lem:weight_between_two_intervals} in the same way that
  equation~\eqref{eq:diff_of_weights} follows from that lemma.  Then
  essentially the same calculations as in~\eqref{eq:LHS_estimate} give
  \begin{multline*}
    \frac{m}{n}\Big(w\big([s,s+\ell),[t,t+\ell)\big)
    -w\big([s+d,s+\ell),[t+d,t+\ell)\big)\Big)\\
    =\frac{|A|}{m}\(2(\ell-d)d+d^2+O_1(2\delta')\(\ell^2+(\ell-d)^2\)\), 
  \end{multline*}
  and hence
  \begin{equation}
    \label{eq:ultimate_lw_bd}
    d\big|A\cap\rho^{-1}\big([d_1,d_2)\big)\big|
    \geq
    \frac{|A|}{m}\(2\ell d-d^2+O_1(2\delta')\(\ell^2+(\ell-d)^2\)\).
  \end{equation}
  Inequality~\eqref{eq:distribution_of_A} follows
  from~\eqref{eq:ultimate_up_bd} and~\eqref{eq:ultimate_lw_bd}, and thus
  Lemma~\ref{lem:distribution_of_A} is proved.
\end{proof}

We prove a simple corollary of Lemma~\ref{lem:distribution_of_A} that allows
us to rewrite the conditions of Lemma~\ref{lem:distribution_of_A} in terms of
$d_1$ and $d_2$. Moreover, the hypotheses of Lemma~\ref{lem:distribution_of_A}
imply that $d_2-d_1=2\ell$ is even. The following corollary overcomes this
shortcoming.

\begin{corollary}
  \label{lem:rewriting}
  Let~$d$, $d_1$, and $d_2$ be positive integers,~$\delta'$ a real number such
  that $0<\delta'\leq 1$, and suppose that
  \begin{flalign*}
    \quad {\rm(\textit{i\/})}   &
    \quad 0< d\leq d_1< d_2-1<d_2+1\leq \frac12m-d\,,&&\\
    \quad {\rm(\textit{ii\/})}  &
    \quad \DISC_2(\delta')\ \text{holds for}\ G=G(\Gamma,A)\,,&&\\
    \quad {\rm(\textit{iii\/})} &
    \quad \frac12(d_2-d_1-1)-d\geq\delta' m\,,&&\\
    \quad {\rm(\textit{iv\/})}  &
    \quad n\sim_{\delta'/2} (n-1)\,.
  \end{flalign*}
  Then 
  \begin{equation}\label{eq:rewrite}
    \left||A\cap\rho^{-1}\big([d_1,d_2)\big)|-\frac{d_2-d_1}{m}|A|\right|
    \leq\frac{|A|}{m}\(d+1+\frac{\delta'}{d}\cdot\frac{m^2}{4}\)\,.
  \end{equation}
\end{corollary}
\begin{proof}
  We distinguish two cases depending on the parity of $d_2-d_1$. We later
  reduce the second case ($d_2-d_1$ odd) to the first case ($d_2-d_1$ even).
  In order to be prepared for this we are going to show a stronger statement
  for the first case.

  \begin{case}[$d_2-d_1$ is even]
    Let us consider the following weaker conditions 
    $\textrm{(\textit{i\/}}'\textrm{)}$ and
    $\textrm{(\textit{iii\/}}'\textrm{)}$ 
    instead of \textrm{(\textit{i})} and \textrm{(\textit{iii})}:
    \begin{flalign*}
      \quad \textrm{(\textit{i\/}}'\textrm{)}  &
      \quad 0< d\leq d_1< d_2\leq \frac{m}{2}-d,&&\\
      \quad \textrm{(\textit{iii\/}}'\textrm{)} &
      \quad \frac12(d_2-d_1)-d\geq\delta' m\,. 
    \end{flalign*}
    We are now going to show a stronger conclusion
    than~\eqref{eq:rewrite} under 
    these weaker assumptions. In particular, we shall verify
    \begin{equation}\label{eq:rewrite2}
      \left||A\cap\rho^{-1}\big([d_1,d_2)\big)|-\frac{d_2-d_1}{m}|A|\right|
      \leq\frac{|A|}{m}\(d+\frac{\delta'}{d}\cdot\frac{m^2}{4}\)\,.
    \end{equation}
    For this we want to apply Lemma~\ref{lem:distribution_of_A} for the
    ``right'' choice of~$s$, $t$, and $\ell$. First, note
    that~\textrm{(\textit{ii})} and~\textrm{(\textit{iv})} are the same in
    Lemma~\ref{lem:distribution_of_A} and Corollary~\ref{lem:rewriting}.
    We set
    \begin{gather*}
      s=d\,,\quad \ell=\frac12(d_2-d_1)\,,\quad \text{and}\quad t=s+\ell+d_1\,.
    \end{gather*}
    Simple calculations using 
    $\textrm{(\textit{i\/}}'\textrm{)}$ and
    $\textrm{(\textit{iii\/}}'\textrm{)}$ show that
    $$
    0=s-d<s+\ell\leq t-d<t+\ell\leq \frac12m\quad \text{and}\quad
    \ell-d\geq\delta' m,
    $$
    and hence~\textrm{(\textit{i})} and~\textrm{(\textit{iii})} of
    Lemma~\ref{lem:distribution_of_A} hold for this particular choice
    of~$s$,~$t$, and~$\ell$. Moreover, $t-s-\ell=d_1$ and $t-s+\ell=d_2$, thus
    Lemma~\ref{lem:distribution_of_A} implies
    $$
     \left||A\cap\rho^{-1}\big([d_1,d_2)\big)|-\frac{2\ell}{m}|A|\right|
    \leq\frac{|A|}{m}\(d+\frac{2\delta'}{d}\((\ell+d)^2+\ell^2\)\)\,,
    $$
    which, combined with 
    $$
    \ell^2\leq(\ell+d)^2=\frac{(d_2-d_1+2d)^2}{4}=
    \frac{\big((d_2+d)+(d-d_1)\big)^2}{4}\leq
    \frac{(m/2)^2}{4}=\frac{m^2}{16},
    $$
    gives inequality~\eqref{eq:rewrite2}.
  \end{case}

  \begin{case}[$d_2-d_1$ is odd] 
    The hypotheses of Lemma~\ref{lem:distribution_of_A}
    unfortunately always imply $\ell=d_2-d_1$ is even. In order to 
    get a bound for intervals $[d_1,d_2)$ of odd length we ``sandwich''
    $|A\cap\rho^{-1}([d_1,d_2))|$ as follows:
    \begin{equation}\label{eq:sandwich}
      |A\cap\rho^{-1}([d_1,d_2-1))|\leq
      |A\cap\rho^{-1}([d_1,d_2))|\leq 
      |A\cap\rho^{-1}([d_1,d_2+1))|\,.
    \end{equation}
    Now we apply Case~1 twice to derive the upper and lower bounds
    in~\eqref{eq:rewrite}.  For the upper bound, we set
    $$
    d_2^{\textsc{U}}=d_2+1\,.
    $$
    Then conditions~\textrm{(\textit{i})} and~\textrm{(\textit{iii})} of
    Corollary~\ref{lem:rewriting}
    are ``strong'' enough to imply
    $\textrm{(\textit{i\/}}'\textrm{)}$ and
    $\textrm{(\textit{iii\/}}'\textrm{)}$ of Case~1, applied to
    $d_2^{\textsc{U}}$ instead of $d_2$. Thus, by~\eqref{eq:rewrite2}, we may
    bound the right-hand side of~\eqref{eq:sandwich} from above by
    \begin{multline*}
      |A\cap\rho^{-1}([d_1,d_2+1))|=|A\cap\rho^{-1}([d_1,d_2^{\textsc{U}}))|\\
      \leq\frac{|A|}{m}\(d+\frac{\delta'}{d}\cdot\frac{m^2}{4}\)+
      \frac{d_2^{\textsc{U}}-d_1}{m}|A|\\
      =\frac{|A|}{m}\(d+1+\frac{\delta'}{d}\cdot\frac{m^2}{4}\)+
      \frac{d_2-d_1}{m}|A|\,.
    \end{multline*}
    Hence, the upper bound for $|A\cap\rho^{-1}([d_1,d_2))|$ 
    in~\eqref{eq:rewrite} follows.
    The lower bound necessary to complete the proof of
    Corollary~\ref{lem:rewriting}
    may be verified by the same
    kind of argument applied to $d_2^{\textsc{L}}=d_2-1$ instead of 
    $d_2^{\textsc{U}}$. \qedhere
  \end{case}
\end{proof}

Corollary~\ref{lem:rewriting} above gives us control over
\begin{equation*}
  \big|A\cap\rho^{-1}\big([d_1,d_2)\big)\big|\,,  
\end{equation*}
as long as~$d_1$ and~$d_2$ are bounded away from~$0$ and $m/2$.
The following two lemmas consider the quantities
\begin{equation}
  \label{eq:end_intervals}
  \big|A\cap\rho^{-1}\big([0,d)\big)\big|
  \quad\text{and}\quad
  \Big|A\cap\rho^{-1}\(\left[\llfloor \frac{m}{2}\rrfloor-d-1,
  \llfloor \frac{m}{2}\rrfloor+1\)\)\Big|.
\end{equation}

\begin{lemma}
  \label{lem:initial_interval}
  Suppose~$0<\delta\leq 1/3$, $d\geq\delta m/2$, and~$n\geq4$.
  If~$\DISC(\delta)$ holds for the Cayley graph~$G=G(\Gamma,A)$, then
  \begin{equation}
    \label{eq:initial_interval}
    \big|A\cap\rho^{-1}\big([0,d)\big)\big|
    \leq 4\frac{d}{m}|A|.
  \end{equation}
\end{lemma}
\begin{proof}
  Let~$\delta>0$ and~$d$ be as in the statement of the lemma, and assume
  that~$G$ satisfies~$\DISC(\delta)$.  Let us estimate from above the number
  of edges induced by
  $$
  U=\rho^{-1}\big([0,2d)\big)
  $$ 
  in~$G$.  We
  have~$|U|=2dn/m\geq\delta n$.  Invoking~$\DISC(\delta)$, using
  that~$0<\delta\leq1/3$ and~$n\geq4$, and recalling that $|A|=\alpha n$,
  we obtain that
  \begin{equation}
    \label{eq:no_edges_upbd}
    e(G[U])\leq(1+\delta)\frac{|A|}{n-1}\binom{|U|}{2}
    \leq4\alpha \(\frac{dn}{m}\)^2.
  \end{equation}
  On the other hand, by Lemma~\ref{lem:weighted_graph}, we have
  \begin{equation}
    \label{eq:other_hand}
      e(G[U])=\sum_{r=0}^{2d-1}\sum_{s=r+1}^{2d-1}w(\{r,s\})+\sum_{r=0}^{2d-1}w(r)
      \geq\sum_{r=0}^{d-1}\sum_{s=r+1}^{r+d-1}w(\{r,s\})+\sum_{r=0}^{2d-1}w(r).
  \end{equation}
  Now fix~$0\leq r<d$.  If~$r<s<r+d$, then~$0<s-r<d$, and,
  by~\eqref{eq:weighted_lemma_diff} of Lemma~\ref{lem:weighted_graph},
  we have
  \begin{align*}
      \sum_{s=r+1}^{r+d-1}w(\{r,s\})
      &=\frac{n}{m}\sum_{s=r+1}^{r+d-1}|A\cap\rho^{-1}(s-r)|\\
      &=\frac{n}{m}\sum_{f=1}^{d-1}|A\cap\rho^{-1}(f)|
      =\frac{n}{m}\big|A\cap\rho^{-1}\big([1,d)\big)\big|\,.
  \end{align*}
  Also, by~\eqref{eq:weighted_lemma_eq},
  $$
  w(r)=\frac{n}{2m}\big|A\cap\rho^{-1}(0)\big|\,,
  $$
  and we 
  may conclude from~\eqref{eq:other_hand} that
  \begin{multline}
    \label{eq:other_hand2}
    e(G[U])
    \geq d\frac{n}{m}\big|A\cap\rho^{-1}\big([1,d)\big)\big|
    +2d\frac{n}{2m}\big|A\cap\rho^{-1}(0)\big|
    =d\frac{n}{m}\big|A\cap\rho^{-1}\big([0,d)\big)\big|.
  \end{multline}
  Comparing~\eqref{eq:no_edges_upbd} and~\eqref{eq:other_hand2},
  inequality~\eqref{eq:initial_interval} follows and our lemma is proved. 
\end{proof}

Our next lemma concerns the second interval in~\eqref{eq:end_intervals}. 

\begin{lemma}
  \label{lem:final_interval}
  Suppose~$0<\delta'\leq1/3$, $d\geq\delta' m$, and~$n\geq4$.
  If~$\DISC_2(\delta')$ holds for the Cayley graph~$G=G(\Gamma,A)$, then
  \begin{align}
    \label{eq:final_interval}
    \Big|A\cap\rho^{-1}\(\left[\llfloor \frac{m}{2}\rrfloor-d-1,
    \llfloor \frac{m}{2}\rrfloor+1\)\)\Big|
    \leq4\frac{d+1}{m}|A|.
  \end{align}
\end{lemma}
\begin{proof}
  Let~$\delta'>0$ and~$d$ be as in the statement of the lemma, and assume
  that~$G$ satisfies~$\DISC_2(\delta')$.  Let us estimate from above the number
  of edges in the bipartite graph induced by the vertex classes
  $$
  U=\rho^{-1}([0,d))\quad\text{and}\quad
  W=\rho^{-1}\(\big[\lfloor m/2\rfloor-d-1,\lfloor m/2\rfloor+d\big)\)
  $$ 
  in~$G$.  We
  have~$|U|=dn/m\geq\delta' n$ and~$|W|=(2d+1)n/m>\delta' n$.
  Invoking~$\DISC_2(\delta')$ and using that~$0<\delta'\leq1/3$ and~$n\geq4$, we
  obtain that
  \begin{equation}
    \label{eq:no_edges_upbd_F}
    e(G[U,W])\leq(1+\delta')\frac{|A|}{n-1}|U||W|
    \leq4\alpha\frac{d(d+1)n^2}{m^2}.
  \end{equation}
  On the other hand, by Lemma~\ref{lem:weighted_graph}, we have
  \begin{equation}
    \label{eq:other_hand_F}
    \begin{split}
      e(G[U,W])
      &\geq\sum_{r=0}^{d-1}
      \sum\left\{w(\{r,s\})
        \:\,\llfloor \frac{m}{2}\rrfloor-d-1+r
        \leq s < \llfloor \frac{m}{2}\rrfloor+1+r\right\}\\
      &=\sum_{r=0}^{d-1}
      \frac{n}{m}\sum\left\{|A\cap\rho^{-1}(f)|
        \:\llfloor \frac{m}{2}\rrfloor-d-1
        \leq f<\llfloor \frac{m}{2}\rrfloor+1\right\}\\
      &=d\frac{n}{m}
      \Big|A\cap\rho^{-1}\(\Big[\llfloor \frac{m}{2}\rrfloor-d-1,
      \llfloor \frac{m}{2}\rrfloor+1\Big)\)\Big|. 
    \end{split}
  \end{equation}
  Comparing~\eqref{eq:no_edges_upbd_F} and~\eqref{eq:other_hand_F}, we
  deduce~\eqref{eq:final_interval}.  Lemma~\ref{lem:final_interval}  is
  proved. 
\end{proof}

We may now prove Lemma~\ref{lem:auxZ}.  

\begin{proof}[Proof of Lemma~\ref{lem:auxZ}]
  Let positive reals $\eta\leq1$ and $\sigma\leq1$ be given.  We may assume
  that~$\sigma\eta\leq1/2$.  We set
  \begin{equation}\label{eq:final_delta}
    \delta'=\frac{(\eta\sigma)^2}{240}\leq\frac{1}{3}\,,
  \end{equation}
  and we let $\delta$ be sufficiently small (as given by
  Fact~\ref{lem:DISC_implies_DISC_2}), so that $\DISC(\delta)$ implies
  $\DISC_2(\delta')$. Again, without loss of generality, we may assume
  $\delta\leq\delta'$. We let~$n_0\geq 4$ be sufficiently large so that
  $n\sim_{\delta'/2}(n-1)$ for every $n\geq n_0$.  We shall now show that this
  choice of constants will do.  Thus, let~$G=G(\Gamma,A)$
  with~$|\Gamma|=n\geq n_0$ and $m>1/\delta$ be given (we follow the notation
  in Setup~\ref{setup:1}), and suppose that~$G$ satisfies~$\DISC(\delta)$.

  Let $0\leq D_1<D_2\leq \lfloor m/2 \rfloor+1$ be such
  that $D_2-D_1\geq \eta m$. We fix
  \begin{equation*}
  d=\left\lfloor\frac{\eta\sigma m}{60}\right\rfloor
  \geq\llfloor\frac{\eta\sigma}{60\delta}\rrfloor\geq
  \llfloor\frac{4}{\eta\sigma}\rrfloor\geq4,
  \end{equation*}
  and 
  $$
  d_1=\max\{D_1,d\}\quad\text{and}\quad 
  d_2=\min\left\{D_2,\llfloor\frac{m}{2}\rrfloor-d-1\right\}\,.
  $$
  Using $m>1/\delta\geq 1/\delta'=240/(\eta\sigma)^2$, $2\eta\sigma\leq1$, and 
  $d+1\leq3d/2$, we see that
  \begin{equation}
    1<\delta m\leq \delta' m \leq
    \frac{\delta' m}{2\eta\sigma}=
    \frac{\eta\sigma m}{480}
    \leq d<d+1\leq\frac{\eta\sigma
      m}{40}\label{eq:final_d}
  \end{equation}
  and
  \begin{equation}
    \label{eq:final_int}
    D_2-D_1-2(d+1)\leq d_2-d_1\leq D_2-D_1\,.
  \end{equation}
  Now we check that the assumptions of
  Corollary~\ref{lem:rewriting}, Lemma~\ref{lem:initial_interval}, and
  Lemma~\ref{lem:final_interval} hold
  simultaneously. It is obvious that the conditions in 
  Lemma~\ref{lem:initial_interval} and~\ref{lem:final_interval} hold by
  our choice of $\delta$, $\delta'$, $n_0$, and the 
  inequalities in~\eqref{eq:final_d}. Moreover,
  conditions~\textrm{(\textit{ii})} 
  and~\textrm{(\textit{iv})} of Corollary~\ref{lem:rewriting} 
  hold by the definition of~$\delta$ (yielding $\DISC_2(\delta')$ for
  $G$), and~$n_0$.
  It remains to verify~\textrm{(\textit{iii})} 
  and~\textrm{(\textit{i})} in Corollary~\ref{lem:rewriting}.
  We start with condition~\textrm{(\textit{iii})}. 
  For this we note that, by the left-hand side of~\eqref{eq:final_int}
  and by~\eqref{eq:final_d},
  $$
  d_2-d_1\geq\eta m-\frac{\eta\sigma m}{20}\,.
  $$
  Using $1<\delta m$,~\eqref{eq:final_d}, and $\delta\leq\delta'\leq\eta/6$
  (see~\eqref{eq:final_delta}),
  we verify \textrm{(\textit{iii})}:
  \begin{align*}
    \frac{1}{2}(d_2-d_1-1)-d&\geq 
    \frac{1}{2}\(\eta m-\frac{\eta\sigma m}{20}
    -\delta m-\frac{\eta\sigma m}{20}\)\\
    &\geq \frac{1}{2}\(\frac{\eta m}{2}- \delta m\)\geq
    \frac{3\delta'm -\delta'm}{2}=\delta'm\,.
  \end{align*}
  Moreover, the last inequality implies $d_2-1-d_1>1$ (using $\delta'm>1$) and  
  thus \textrm{(\textit{i})} of Corollary~\ref{lem:rewriting} follows as well.

  Having verified the assumptions of
  Corollary~\ref{lem:rewriting}, Lemma~\ref{lem:initial_interval}, and
  Lemma~\ref{lem:final_interval}, we 
  use these lemmas to verify the claim of Lemma~\ref{lem:auxZ}, i.e.,
  \begin{equation}\label{eq:final_goal}
    \big|A\cap\rho^{-1}\big([D_1,D_2)\big)\big|\sim_{\sigma}\frac{D_2-D_1}{m}|A|\,.
  \end{equation}
  We first derive the upper bound in~\eqref{eq:final_goal}.
  Note that
  \begin{align*}
    \big|A\cap\rho^{-1}\big([D_1,D_2)\big)\big|
   &\leq 
    \big|A\cap\rho^{-1}\big([0,d)\big)\big|+
    \big|A\cap\rho^{-1}\big([d_1,d_2)\big)\big|\\
    &\qquad\qquad+
    \big|A\cap\rho^{-1}\big([\lfloor m/2\rfloor-d-1,\lfloor m/2\rfloor+1)\big)\big|\,.
  \end{align*}
  Applying Lemma~\ref{lem:initial_interval},
  Corollary~\ref{lem:rewriting}, and Lemma~\ref{lem:final_interval} to the
  first, second, and third terms, of the right-hand side of the above inequality,
  respectively, yields
  \begin{align*}
    \big|A\cap\rho^{-1}\big([D_1,D_2)\big)\big|
    &\leq 
    4\frac{d}{m}|A| + 
    \(\frac{|A|}{m}\(d+1+\frac{\delta' m^2}{4d}\)+\frac{d_2-d_1}{m}|A|\)+
    4\frac{d+1}{m}|A|\\
    &\leq \frac{|A|}{m}\(10(d+1) +\frac{\delta' m^2}{4d}\)+\frac{d_2-d_1}{m}|A|\,.
  \end{align*}
  Using~\eqref{eq:final_delta}, \eqref{eq:final_d}, and~\eqref{eq:final_int}, 
  we can bound this last expression further by
  $$
  \frac{|A|}{m}\(\frac{\sigma \eta m}{4}+\frac{\sigma
    \eta m}2\)+\frac{D_2-D_1}{m}|A|\,,
  $$
  and, finally, $D_2-D_1\geq\eta m$ gives
  \begin{equation}\label{eq:final_upper}
    \big|A\cap\rho^{-1}\big([D_1,D_2)\big)\big| 
    \leq\(1+\sigma\)\frac{D_2-D_1}{m}|A|\,.
  \end{equation}
  
  It is left for us to show the lower bound in~\eqref{eq:final_goal}.  Note
  $$
  \big|A\cap\rho^{-1}\big([D_1,D_2)\big)\big|\geq 
  \big|A\cap\rho^{-1}\big([d_1,d_2)\big)\big|\,,
  $$
  and hence Corollary~\ref{lem:rewriting} implies
  \begin{equation}\label{eq:final_low1}
    \big|A\cap\rho^{-1}\big([D_1,D_2)\big)\big|
    \geq \frac{d_2-d_1}{m}|A|-\frac{|A|}{m}
    \(d+1+\frac{\delta' m^2}{4d}\)\,.
  \end{equation}
  Similar calculations to the ones above, based on~\eqref{eq:final_int},
  \eqref{eq:final_delta}, and~\eqref{eq:final_d}, show that
  \begin{align} 
      \frac{d_2-d_1}{m}|A|-\frac{|A|}{m}
      \!\!\(d+1+\frac{\delta' m^2}{4d}\)\!\!
      &\geq\frac{D_2-D_1}{m}|A|-\frac{|A|}{m}\(3(d+1)+\frac{\delta'
        m^2}{4d}\)\nonumber\\
      &\geq \frac{D_2-D_1}{m}|A|\!\!\(\!1\!-\!\frac{\sigma\eta
        m}{10 (D_2-D_1)}\!-\!
      \frac{\sigma\eta m}{2 (D_2-D_1)}\!\)\!\!.\label{eq:final_low2} 
  \end{align}
  Again, since $D_2-D_1\geq\eta m$, it follows
  from~\eqref{eq:final_low1} 
  combined with~\eqref{eq:final_low2}
  that
  \begin{equation}\label{eq:final_lower}
    \big|A\cap\rho^{-1}\big([D_1,D_2)\big)\big|
    \geq\(1-\sigma\)\frac{D_2-D_1}{m}|A|\,.
  \end{equation}
  Finally,~\eqref{eq:final_upper} and~\eqref{eq:final_lower}
  imply~\eqref{eq:final_goal}, and therefore Lemma~\ref{lem:auxZ} is proved.
\end{proof}

\subsubsection{$\ZINTDISC$ implies $\EIG$}\label{sec:ZINTDISC_to_EIG}
In this section we give the proofs of Lemmas~\ref{lem:translation}
and~\ref{lem:auxS}. We start with the proof of Lemma~\ref{lem:translation},
which ``translates'' the results of Lemma~\ref{lem:auxZ} for $\rho$ and
$\ZZ/m\ZZ$ to $\chia$ and $\RR/2\pi\RR$.

\begin{proof}[Proof of Lemma~\ref{lem:translation}]
Let~$\sigma$ and~$\eta$ be as in the statement of Lemma~\ref{lem:translation},
and suppose~$A$ satisfies $\ZINTDISC(\rho;\eta/2,\sigma/3)$.
Let~$0\leq\theta_1<\theta_2\leq\pi$ with~$\theta_2-\theta_1\geq2\pi\eta$ be
given.  Our aim is to show~\eqref{eq:S-INT-DISC}, i.e.,
$$
\big|A\cap\chia^{-1}\([\theta_1,\theta_2]\)\big|
\sim_{\sigma}\frac{\theta_2-\theta_1}{2\pi}|A|\,.
$$

Recall we have~$\rho\:\Gamma\to\ZZ/m\ZZ$ and~$\Omega\:\ZZ/m\ZZ\to \RR/2\pi\RR$ such
that~$\chia=\Omega\rho$ (see Setup~\ref{setup:1}).  Put
\begin{equation*}
  \label{eq:def_d1_d2}
  D_1=\llceil m\frac{\theta_1}{2\pi}\rrceil
  \quad\text{and}\quad
  D_2=\llfloor m\frac{\theta_2}{2\pi}\rrfloor+1\,.
\end{equation*}
Observe that
$$
\frac{2\pi}m\(\llceil\frac{m\theta_1}{2\pi}\rrceil-1\)
<\theta_1\leq
\frac{2\pi}m\llceil\frac{m\theta_1}{2\pi}\rrceil,
$$
and hence we have
$$
\frac{2\pi}m\(\llceil\frac{m\theta_1}{2\pi}\rrceil-1\)=
\Omega(D_1-1)<\theta_1\leq\Omega(D_1)
=\frac{2\pi}m\llceil\frac{m\theta_1}{2\pi}\rrceil\,.
$$
Similarily, one may check that
$\Omega(D_2-1)\leq\theta_2<\Omega(D_2)$,
and consequently 
\begin{equation*}
  \label{eq:Omega_back}
  \Omega^{-1}\big([\theta_1,\theta_2]\big)
  =\left[D_1,D_2\).
\end{equation*}
We now observe that
\begin{equation*}
  \label{eq:ZINT-condition}
  D_2-D_1=\llfloor m\frac{\theta_2}{2\pi}\rrfloor+1
  -\llceil m\frac{\theta_1}{2\pi}\rrceil
  =\frac{m}{2\pi}(\theta_2-\theta_1)+O_1(1).
\end{equation*}
Using that~$m\eta\geq m\eta\sigma\geq 3$, we deduce that 
\begin{equation*}
  \label{eq:ZINT-condition2}  
  D_2-D_1
  \geq\frac{m}{2\pi}(\theta_2-\theta_1)-1
  \geq\eta m-1\geq\frac{1}{2}\eta m.
\end{equation*}
Hence, by property~$\ZINTDISC(\rho;\eta/2,\sigma/3)$, we have
\begin{equation}
  \label{eq:pf_translation}
  \begin{split}
    \big|A\cap\chia^{-1}([\theta_1,\theta_2])\big|
    &=\big|A\cap\rho^{-1}([D_1,D_2))\big|\\
    &\sim_{\sigma/3}\frac{1}{m}(D_2-D_1)|A|\\
    &=\frac{1}{m}\(m\frac{\theta_2}{2\pi}-m\frac{\theta_1}{2\pi}+O_1(1)\)|A|\\
    &=\frac{1}{2\pi}\(\theta_2-\theta_1+O_1\(\frac{2\pi}{m}\)\)|A|,
  \end{split}
\end{equation}
which, using that~$m\eta\sigma\geq 3$, is
\begin{equation}
  \label{eq:pf_translation2}
  \frac{1}{2\pi}\(1+O_1\(\frac{\sigma}{3}\)\)(\theta_2-\theta_1)|A|.
\end{equation}
We conclude from~\eqref{eq:pf_translation} and~\eqref{eq:pf_translation2} that
\begin{equation*}
  \label{eq:pf_translation3}
  \big|A\cap\chia^{-1}([\theta_1,\theta_2])\big|
  \sim_\sigma\frac{(\theta_2-\theta_1)}{2\pi}|A|,
\end{equation*}
as required.  The proof of Lemma~\ref{lem:translation} is complete.
\end{proof}

Finally, we prove the last auxiliary lemma, Lemma~\ref{lem:auxS}, 
used in the proof of the
main theorem, Theorem~\ref{thm:main_positive}. 
\begin{proof}[Proof of Lemma~\ref{lem:auxS}]
Let~$0<\eps\leq1$ be given.  We define the constants~$\eta$ and~$\sigma>0$
as follows.
\begin{equation}
  \label{eq:def_eta_sigma}
  \eta=\frac{\eps}{8\pi}<\frac{1}{16}
  \quad\text{and}\quad
  \sigma=\frac{1}{8}\eps.
\end{equation}
In the remainder of the proof, we show that the above choice for the
constants~$\eta$ and~$\sigma>0$ will do.  Thus, let us suppose that
the set~$A\subseteq\Gamma\setminus\{0\}$
satisfies property $\SINTDISC(\chia;\eta,\sigma)$.  Our aim is to show
that~\eqref{eq:auxS} holds, i.e.,
$$
|\lambda^{(\chi)}|=\Big|\sum_{a\in A}\chi(a)\Big|\leq\eps|A|\,.
$$

For a complex number $z\in\CC$, in what follows, we write~$\Re(z)$ for the
real and~$\Im(z)$ for the imaginary part of~$z$.  Let us first observe that,
owing to the fact that the eigenvalues of an undirected graph are real and
that~$A=-A$, we have
\begin{equation}\label{eq:pf_auxS}
  \begin{split}
    \lambda^{(\chi)}&=\sum_{a\in A}\chi(a)
    =\Re\bigg(\sum_{a\in A}\chi(a)\bigg)\\
    &=2\Re\bigg(\sum_{a\in A}\{\chi(a)\colond \Im \(\chi(a)\)>0\}\bigg)+
    \Re\bigg(\sum_{a\in A}\{\chi(a)\colond \Im\(\chi(a)\)=0\}\bigg).
  \end{split}
\end{equation}
Moreover, it follows from $\SINTDISC(\chia;\eta,\sigma)$ that
\begin{align*}
|A\cap\chi^{-1}(1)|=|A\cap\chia^{-1}(0)|&\leq|A\cap\chia^{-1}([0,2\pi\eta])|\\
&\leq(1+\sigma)\eta|A|\leq 2\eta|A|\leq\frac{1}{4}\eps|A|\,.
\end{align*}
Similarly, we observe that $|A\cap\chi^{-1}(-1)|=|A\cap\chia^{-1}(\pi)|\leq
(1/4)\eps|A|$ and thus
$$
\Bigg|\Re\bigg(\sum_{a\in A}\{\chi(a)\colond
\Im\(\chi(a)\)=0\}\bigg)\Bigg|
=\big||A\cap\chi^{-1}(1)|-|A\cap\chi^{-1}(-1)|\big|\leq\frac{1}{4}\eps |A|\,.
$$
Therefore, we infer from~\eqref{eq:pf_auxS} that
\begin{equation}\label{eq:auxS_f}
  |\lambda^{(\chi)}|\leq 2\bigg|\Re\bigg(\sum_{a\in A}
  \{\chi(a)\colond \Im \(\chi(a)\)\geq 0\}\bigg)\bigg|+\frac{1}{4}\eps |A|\,.
\end{equation}
Thus, we are interested in~$\Re(\chi(a))$ ($a\in A$, $\Im(\chi(a))\geq 0$).

Put
\begin{equation}
  \label{eq:def_k_theta}
  k=\llfloor\frac{\pi/2}{4\pi\eta}\rrfloor
  =\llfloor\frac{1}{8 \eta}\rrfloor
  \quad\text{and}\quad
  \phi=\frac{\pi}{2k}.
\end{equation}
Observe for later reference that by~\eqref{eq:def_eta_sigma}
and~\eqref{eq:def_k_theta}
\begin{equation}
  \label{eq:phi_bds}
  4\pi\eta\leq\phi\leq 8\pi\eta=\eps.
\end{equation}
We shall subdivide the upper half of~$S^1$ into~$2k$ arcs, symmetric with
respect to the imaginary axis (in fact, $\bfi$~will be left out).  The
endpoints of the arcs will be~$0,\phi,2\phi,\dots,\pi$.  Let us in fact denote
by~$I_j^+$ the arc of the~$z=\exp(\theta\bfi)\in S^1$ with
\begin{equation}
  \label{eq:def_Ij+}
  (j-1)\phi\leq\theta<j\phi,
\end{equation}
for all~$1\leq j\leq k$.  Similarly, we let~$I_j^-$ be the arc of
the~$z=\exp(\theta\bfi)\in S^1$ with
\begin{equation}
  \label{eq:def_Ij-}
  \pi-j\phi<\theta\leq\pi-(j-1)\phi,
\end{equation}
for all~$1\leq j\leq k$.  Clearly, if~$a\in\Gamma$ is such that~$\chi(a)\in
I_j^+$ ($1\leq j\leq k$), then
\begin{equation}
  \label{eq:cos_bds_Ij+}
  \cos (j\phi)<\Re(\chi(a))\leq\cos((j-1)\phi)\,.
\end{equation}
Similarly, if~$a\in\Gamma$ is such that~$\chi(a)\in I_j^-$ ($1\leq j\leq k$),
then
\begin{equation}
  \label{eq:cos_bds_Ij-}
  -\cos((j-1)\phi)<\Re(\chi(a))\leq-\cos (j\phi)\,.
\end{equation}
We now state and prove the following claim.

\begin{claim}
  \label{claim:SINT_variant}
  For all~$1\leq j\leq k$ and both~$*\in\{+,-\}$, we have
  \begin{equation}
    \label{eq:claim_SINT_variant}
    \big|A\cap\chi^{-1}(I_j^*)\big|\sim_{2\sigma}\frac{\phi}{2\pi}|A|. 
  \end{equation}
\end{claim}
\begin{proof}
  This claim follows easily from $\SINTDISC(\chia;\eta,\sigma)$.
  Let~$0\leq\psi_1<\psi_2\leq\pi$ be such that
  \begin{equation}
    \label{eq:psi_gap_bd}
    \psi_2-\psi_1\geq4\pi\eta.
  \end{equation}
  We shall show that
  \begin{equation}
    \label{eq:psi_interval}
    \big|A\cap\chia^{-1}\big([\psi_1,\psi_2)\big)\big|
    \sim_{2\sigma}\frac{1}{2\pi}(\psi_2-\psi_1)|A|.
  \end{equation}
  One may similarly show that 
  \begin{equation}
    \label{eq:psi_interval2}
    \big|A\cap\chia^{-1}\big((\psi_1,\psi_2]\big)\big|
    \sim_{2\sigma}\frac{1}{2\pi}(\psi_2-\psi_1)|A|.
  \end{equation}
  Claim~\ref{claim:SINT_variant} follows from~\eqref{eq:psi_interval}
  and~\eqref{eq:psi_interval2}. In particular~\eqref{eq:psi_interval}
  with $\psi_1=(j-1)\phi$ and $\psi_2=j\phi$ gives
  the claim for intervals of the type $I^+_j$ and
  similarly~\eqref{eq:psi_interval2}
  yields~\eqref{eq:claim_SINT_variant} for intervals of the type~$I^-_j$.

  To prove~\eqref{eq:psi_interval}, observe first that there is a~$\xi>0$ such
  that 
  \begin{equation}
    \label{eq:equivalent}
    A\cap\chia^{-1}\big([\psi_1,\psi_2)\big)
    =A\cap\chia^{-1}\big([\psi_1,\psi_2-\xi]\big).
  \end{equation}
  Moreover, we may clearly assume that
  \begin{equation}
    \label{eq:xi_bound}
    \frac{\xi}{4\pi\eta}\leq\frac{\sigma}{2}.
  \end{equation}
  Then~$\psi_2-\psi_1-\xi\geq2\pi\eta$, and hence we may apply
  $\SINTDISC(\chia;\eta,\sigma)$ to deduce that
  \begin{equation}
    \label{eq:psi_bound_final}
    \big|A\cap\chia^{-1}\big([\psi_1,\psi_2-\xi]\big)\big|
    \sim_\sigma\frac{1}{2\pi}(\psi_2-\psi_1-\xi)|A|,
  \end{equation}
  and, by~\eqref{eq:psi_gap_bd} and~\eqref{eq:xi_bound}, we have 
  \begin{equation}
    \label{eq:psi_bound_final2}
    \frac{1}{2\pi}(\psi_2-\psi_1-\xi)|A|
    \sim_{\sigma/2}\frac{1}{2\pi}(\psi_2-\psi_1)|A|. 
  \end{equation}
  Relation~\eqref{eq:psi_interval} follows from~\eqref{eq:equivalent},
  \eqref{eq:psi_bound_final}, and~\eqref{eq:psi_bound_final2}, by
  $$
  A\cap\chia^{-1}\big([\psi_1,\psi_2)\big)
  =A\cap\chia^{-1}\big([\psi_1,\psi_2-\xi]\big)\sim_{2\sigma}
  \frac{1}{2\pi}(\psi_2-\psi_1)|A|\,.
  $$
\end{proof}

Claim~\ref{claim:SINT_variant} and inequalities~\eqref{eq:cos_bds_Ij+}
and~\eqref{eq:cos_bds_Ij-} may now be used to estimate the sum
in~\eqref{eq:auxS_f}.  We have
\begin{multline}
  \label{eq:split_sum_j}
  \Re\bigg(\sum_{a\in A}\{\chi(a)\:\Im\(\chi(a)\)\geq 0\}\bigg)
  =\Re\bigg(\sum_{j=1}^{k}
  \sum_{a\in A}\{\chi(a)\:\chi(a)\in I_j^+\cup I_j^-\}\bigg)\,. 
\end{multline}
Fix~$1\leq j\leq k$.  We have, by Claim~\ref{claim:SINT_variant} and
inequalities~\eqref{eq:cos_bds_Ij+} and~\eqref{eq:cos_bds_Ij-},
\begin{equation}
  \label{eq:j_sum}
  \Re\!\bigg(\!\sum_{a\in A}\{\chi(a)\:\chi(a)\in I_j^+\cup I_j^-\}\!\!\bigg)
  \leq\frac{\phi}{2\pi}|A|\Big(\!(1+2\sigma)\cos\((j-1)\phi\)
  -(1-2\sigma)\cos\(j\phi\)\!\!\Big)\,.
\end{equation}
Using that~$\cos((j-1)\phi)-\cos(j\phi)\leq\phi$, we observe that the
right-hand side of~\eqref{eq:j_sum} is smaller than
\begin{equation}\label{eq:final_ff}
  \frac{\phi}{2\pi}|A|\Big(\phi+4\sigma\Big)\,.
\end{equation}

Therefore, from~\eqref{eq:j_sum} and~\eqref{eq:final_ff} we deduce
that the expression in~\eqref{eq:split_sum_j} is, in absolute value, at most 
\begin{equation}
  \label{eq:eq:split_sum_j_final}
  k\frac{\phi}{2\pi}|A|(\phi+4\sigma)
  =\frac{1}{4}|A|(\phi+4\sigma)
  \leq\frac{1}{4}|A|\(\eps+\frac{1}{2}\eps\)
  =\frac{3}{8}\eps|A|\,,
\end{equation}
where in this inequality we used~\eqref{eq:def_eta_sigma}
and~\eqref{eq:phi_bds}.  Combining~\eqref{eq:eq:split_sum_j_final}
with~\eqref{eq:auxS_f}, we infer that
$$
|\lambda^{(\chi)}|\leq2\frac{3}{8}\eps|A|+\frac{1}{4}\eps|A|=\eps|A|\,,
$$
as claimed in Lemma~\ref{lem:auxS}.
\end{proof}

\end{document}